\documentclass{birkjour}
 \newtheorem{thm}{Theorem}[section]
 \newtheorem{cor}[thm]{Corollary}
 \newtheorem{lem}[thm]{Lemma}
 
 \theoremstyle{definition}
 \newtheorem{defn}[thm]{Definition}
 \theoremstyle{remark}
 \newtheorem{rem}[thm]{Remark}
 
 \numberwithin{equation}{section}
\usepackage{hyperref}
\begin{document}

%
%
%
%
%
%
%
%
%

\title[The Hamburger Criterion]
 {The Hamburger Criterion for Matrix\\ Nevanlinna–Pick Interpolation}

\author[Yury M. Dyukarev ]{Yury M. Dyukarev}

\address{%
Department of Physics,\\
 V. N. Karazin Kharkiv National University, \\
 Svobody sq. 4,\\
  Kharkiv, 61022, Ukraine \\
}

\email{yu.dyukarev@karazin.ua}




\dedicatory{Dedicated to the memory of my teacher, V.~E.~Katsnelson}

\begin{abstract}
The classical Hamburger moment problem can be viewed as an interpolation problem for Nevanlinna functions. A closely related problem is the Nevanlinna--Pick interpolation problem. Each of these problems is either determinate, with a unique solution, or indeterminate, with infinitely many solutions. Hamburger established a criterion for the indeterminacy of the moment problem in terms of the convergence of two  series involving the classical polynomials of the first and second kind. In this paper, we establish an analogous criterion for the matrix Nevanlinna--Pick interpolation problem. The criterion is formulated in terms of the convergence of matrix series involving the rational functions of the first and second kind.
\end{abstract}

\maketitle

\section{Introduction}

Let
\(
\mathbb{C}_+ = \{z \in \mathbb{C} : \operatorname{Im} z > 0\}
\)
denote the open upper half-plane.
A holomorphic function \(w:\mathbb{C}_+\to\mathbb{C}\) is called a \emph{Nevanlinna function} if
\[ \operatorname{Im} w(z) \ge0
\qquad\text{for every } z\in\mathbb{C}_+.
\]
The class of all such functions is denoted by~$\mathcal{R}$.

In the \emph{Nevanlinna--Pick interpolation problem}, one seeks to describe all Nevanlinna functions \(w\in\mathcal{R}\) satisfying the interpolation conditions
\begin{equation}\label{NPs}
w(z_j)=w_j
\qquad\text{for all } j\in\mathbb{N},
\end{equation}
where the sequence of distinct interpolation nodes
\(
\{z_j\}_{j=1}^{\infty}\subset\mathbb{C}_+
\)
and the corresponding sequence of interpolation values
\(
\{w_j\}_{j=1}^{\infty}\subset\mathbb{C}
\)
are given.

In the \emph{Hamburger moment problem}, a sequence
\(
\{s_j\}_{j=0}^{\infty}\subset\mathbb{R}
\)
is given, and one seeks to describe all non-negative measures $\sigma$ on $\mathbb{R}$ such that
\begin{equation}\label{HMs}
s_j=\int_{\mathbb{R}} t^j\,\sigma(dt)
\qquad\text{for all } j\ge0.
\end{equation}
Any non-negative measure $\sigma$ satisfying~\eqref{HMs} is called a \emph{solution} of the Hamburger moment problem.

The following theorem establishes the equivalence between the Hamburger moment problem and a certain interpolation problem for Nevanlinna functions (see~\cite{Akh}).

\begin{thm}[Hamburger--Nevanlinna]\label{HN}
Let $\sigma$ be a solution of the Hamburger moment problem~\eqref{HMs}, and define
\begin{equation}\label{AF1}
w(z)=\int_{-\infty}^{+\infty}\frac{\sigma(dt)}{t-z}
\qquad\text{for every } z\in\mathbb{C}_+.
\end{equation}
Then \(w\in\mathcal{R}\), and for every \(\delta\in(0,\pi/2)\) the following asymptotic expansion holds:
\begin{equation}\label{ass}
w(z)\sim
-\frac{s_0}{z}
-\frac{s_1}{z^2}
-\cdots
-\frac{s_n}{z^{n+1}}
-\cdots,
\qquad
z\to\infty,\quad 
z\in\Gamma_\delta,
\end{equation}
where
\[
\Gamma_\delta=
\{z\in\mathbb{C}_+:\delta<\operatorname{Arg}z<\pi-\delta\},
\]
and $\operatorname{Arg}z\in(-\pi,\pi]$ denotes the principal value of the argument of the complex number~$z\neq 0$.

Conversely, suppose that \(w\in\mathcal{R}\) and that the asymptotic expansion~\eqref{ass} holds for some sequence of real numbers \(\{s_j\}_{j=0}^{\infty}\), at least along the imaginary axis \(z=iy\) as \(y\to+\infty\).
Then \(w\) admits the representation~\eqref{AF1}, and the coefficients in~\eqref{ass} satisfy
\[
s_j=\int_{\mathbb{R}} t^j\,d\sigma(t)
\qquad\text{for all } j\ge0,
\]
where \(\sigma\) is the measure appearing in~\eqref{AF1}.
\end{thm}

It follows from Theorem~\ref{HN} that the Hamburger moment problem is equivalent to finding a function \(w\in\mathcal{R}\) with a prescribed asymptotic expansion~\eqref{ass}.
This expansion may be viewed as an analogue of the Nevanlinna--Pick interpolation conditions~\eqref{NPs}.
For simplicity, we will refer to problem~\eqref{ass} as the Hamburger moment problem.

The problems~\eqref{NPs} and~\eqref{ass} are both interpolation problems for Nevanlinna functions.
This explains why many results for the two problems are analogous.
At the same time, in the Nevanlinna--Pick problem~\eqref{NPs}, the interpolation nodes are distributed throughout the upper half-plane~\(\mathbb{C}_+\), whereas in the moment problem~\eqref{ass} there is only a single multiple interpolation node at the point \(z=\infty\).
This distinction explains why analogous results for the two problems may take substantially different forms.

We restrict our attention to the case in which each of the problems~\eqref{NPs} and~\eqref{ass} has at least one solution.
Then exactly one of the following alternatives holds for each problem:
either there exist infinitely many solutions (the \emph{indeterminate} case), or there exists a unique solution (the \emph{determinate} case).

Hamburger~\cite{H1,H2,H3} established a criterion for the indeterminacy of the Hamburger moment problem~\eqref{ass}.
The criterion is formulated in terms of the convergence of a  series whose entries are given by the classical polynomials of the first and second kind. A modern operator-theoretic formulation of these results of Hamburger can be found in \cite{Sim}.

In the present paper, we establish an analogous criterion for the indeterminacy of the Nevanlinna--Pick problem~\eqref{NPs}.
The criterion is likewise formulated in terms of the convergence of a series whose entries are the rational functions of the first and second kind (see~\cite{D22}).
In the Nevanlinna--Pick problem, these rational functions play a role analogous to that of the polynomials of the first and second kind in the Hamburger moment problem.

In contrast, the indeterminacy criterion for the Nevanlinna--Pick problem~\eqref{NPs} is considerably more intricate than the corresponding criterion for the Hamburger moment problem~\eqref{ass}.
We show that, for a special class of Nevanlinna--Pick problems with purely imaginary interpolation nodes, the indeterminacy criterion is completely analogous to that for the Hamburger moment problem.

The new results presented in this paper provide a clearer understanding of the similarities and differences between the Nevanlinna--Pick interpolation problem~\eqref{NPs} and the Hamburger moment problem~\eqref{ass} (see also~\cite{D22,D25}).
All results in this paper are established for matrix interpolation problems.
In this introduction, we restrict our discussion to the scalar case for simplicity.

\section{The Matrix Hamburger Moment Problem}

In this section, we summarize several known results concerning the matrix Hamburger moment problem. These results are formulated within the framework of V. P. Potapov's approach to interpolation problems in analysis (see, e.g., \cite{Ko,KPW,KoPo,Sah}). They are stated in a form that makes the analogies between the moment problem and the Nevanlinna--Pick problem transparent.

We begin by introducing the notation used throughout the paper.

Let $\mathbb{C}^{m}$ denote the vector space of complex column vectors,
\[
x=\operatorname{col}(x_1,x_2,\dots,x_m),
\]
equipped with the inner product
\(
(x,y)=\sum_{j=1}^{m}\overline{x_j}\,y_j.
\)
Let $\mathbb{C}^{m\times n}$ denote the space of complex $m\times n$ matrices.
For a matrix $A\in\mathbb{C}^{m\times n}$, its conjugate transpose is denoted by $A^*\in\mathbb{C}^{n\times m}$.
If $A\in\mathbb{C}^{m\times m}$, then $\operatorname{tr}A$ denotes its trace.
A matrix $A\in\mathbb{C}^{m\times m}$ is called \emph{Hermitian} if $A=A^*$.
The set of all Hermitian matrices is denoted by $\mathbb{C}_H^{m\times m}$.
A matrix $A\in\mathbb{C}_H^{m\times m}$ is called \emph{non-negative} if
\(
(x,Ax)\ge0,\ x\in\mathbb{C}^m,
\)
and the set of all such matrices is denoted by $\mathbb{C}^{m\times m}_{\ge}$.
If
\(
(x,Ax)>0
\)
for every nonzero vector $x\in\mathbb{C}^m$, then $A$ is called \emph{positive definite}; the set of all positive definite matrices is denoted by $\mathbb{C}^{m\times m}_{>}$.
The identity matrix of order $m$ is denoted by $I_m$, and the zero matrix of size $m\times n$ by $O_{m\times n}$.
Whenever the dimensions are clear from the context, the subscripts are omitted.
For Hermitian matrices $A,B\in\mathbb{C}_H^{m\times m}$, we write
\(
A\ge B
\)
if $A-B$ is non-negative, and
\(
A>B
\)
if $A-B$ is positive definite.
If $A$ is invertible, we set
\(
A^{-*}=(A^{-1})^*.
\)
For a matrix function $f(z)$, we write
\(
f^*(z)=(f(z))^*,
\)
and, whenever $f(z)$ is invertible,
\(
f^{-*}(z)=\bigl(f(z)^{-1}\bigr)^*.
\)

Throughout the paper, limits involving matrices are understood componentwise or, equivalently, with respect to any matrix norm. Since the matrix spaces under consideration are finite-dimensional, these notions of convergence are equivalent.

Let $\mathfrak{B}$ denote the $\sigma$-algebra of Borel subsets of $\mathbb{R}$.
A set function
\(
\sigma:\mathfrak{B}\to\mathbb{C}^{m\times m}_{\ge}
\)
is called a \emph{non-negative matrix measure} if it is countably additive in the sense that
\[
\sigma\!\left(\bigcup_{j=1}^{\infty}A_j\right)
=\sum_{j=1}^{\infty}\sigma(A_j),
\]
whenever $\{A_j\}_{j=1}^{\infty}$ is a sequence of pairwise disjoint Borel subsets of $\mathbb{R}$.

Finally, we introduce the open upper and lower half-planes:
\[
\mathbb{C}_-=\{z\in\mathbb{C}:\operatorname{Im}z<0\},\qquad
\mathbb{C}_+=\{z\in\mathbb{C}:\operatorname{Im}z>0\}.
\]

In the \emph{matrix Hamburger moment problem}, one is given a sequence of matrices
\begin{equation}\label{s_j}
\{s_j\}_{j=0}^\infty \subset \mathbb{C}^{m \times m}_H,
\end{equation}
and one seeks to describe all non-negative matrix measures $\sigma$ on $\mathbb{R}$ such that
\begin{equation}\label{HM}
  s_j = \int_{\mathbb{R}} t^j \, \sigma(dt) \qquad \text{for each } j \geq 0.
\end{equation}
Any non-negative matrix measure $\sigma$ satisfying \eqref{HM} is called a \emph{solution} of the matrix Hamburger moment problem.

It is well known that a necessary and sufficient condition  for the existence of a solution is that the \emph{Hankel block matrices}
\begin{equation*}
\mathbf{H}_j =
\begin{pmatrix}
s_0 & s_1 & \cdots & s_j \\
s_1 & s_2 & \cdots & s_{j+1} \\
\vdots & \vdots & \ddots & \vdots \\
s_j & s_{j+1} & \cdots & s_{2j}
\end{pmatrix}
\end{equation*}
are non-negative for each $j \geq 0$.

In this paper, we assume that the Hankel block matrices are positive definite:
\begin{equation}\label{PMH}
\mathbf{H}_j > O_{m(j+1)} \qquad \text{for each } j \geq 0.
\end{equation}
These assumptions exclude the degenerate case of the truncated matrix Hamburger moment problem.

Let $\mathcal{M}$ denote the set of all solutions $\sigma$ to problem~\eqref{HM}. Under the above assumption, we have $\mathcal{M} \neq \emptyset $.

\begin{defn}
A holomorphic matrix function \( w: \mathbb{C}_+ \to \mathbb{C}^{m \times m} \)
is called a \emph{Nevanlinna matrix function} if 
\[
\frac{w(z) - w^*(z)}{2i} \geq O \qquad \text{for all } z \in \mathbb{C}_+.
\]
\end{defn}
\noindent
The class of such matrix functions of fixed order $m$ is denoted by $\mathcal{R}_m$.

To each matrix measure $\sigma \in \mathcal{M}$ we associate the matrix function
\begin{equation}\label{AssF}
w(z) = \int_{\mathbb{R}} \frac{\sigma(dt)}{t - z} \qquad \text{for each } z \in \mathbb{C}_+.
\end{equation}
The function thus obtained belongs to the class $\mathcal{R}_m$ and is said to be \emph{associated with the moment problem}~\eqref{HM}.  
Let $\mathcal{F}$ denote the set of all such associated functions.

By the Stieltjes inversion formula, the correspondence between $w\in \mathcal{F}$ and $\sigma \in \mathcal{M}$ established by formula~\eqref{AssF} is bijective.  
Therefore, instead of describing the set of solutions to the moment problem $\mathcal{M}$, we may describe the corresponding set of associated functions $\mathcal{F}$.

Let us fix a point $z_0$ in the upper half-plane $\mathbb{C}_+$ and consider the set
\begin{equation}\label{1.8}
\mathfrak{K}(z_0) = \{w(z_0) : w \in \mathcal{F} \}. 
\end{equation}
There exist matrices $r(z_0) \in \mathbb{C}^{m\times m}_\geq$, 
$\rho(z_0) \in \mathbb{C}^{m\times m}_\geq$, and 
$c(z_0) \in \mathbb{C}^{m\times m}$ such that \eqref{1.8} admits the representation
\begin{equation}\label{1.9}
\mathfrak{K}(z_0) = \left\{ c(z_0) + r(z_0) V \rho(z_0) : V^* V \leq I_m \right\}. 
\end{equation}
The set $\{ V : V^* V \leq I_m \}$ is called the unit matrix ball.
Equation \eqref{1.9}  shows that $\mathfrak{K}(z_0)$ is the image of the unit matrix ball under an affine transformation.  
Accordingly, $\mathfrak{K}(z_0)$ can be interpreted as a matrix ball with \emph{center} $c(z_0)$, 
\emph{left semi-radius} $r(z_0)$, and \emph{right semi-radius} $\rho(z_0)$.
In the context of the matrix Hamburger moment problem, this set is called the \emph{limit matrix Weyl ball} at the point $z_0$.

A theorem of S.~A.~Orlov shows that the ranks of the left and right semi-radii of the limit matrix Weyl balls are constant throughout the upper half-plane. 
In other words, for all $z_1, z_2 \in \mathbb{C}_+$,  
\[
 \operatorname{rank} r(z_1) = \operatorname{rank} r(z_2), \qquad
 \operatorname{rank} \rho(z_1) = \operatorname{rank} \rho(z_2).
\]  
We denote these numbers by $\delta_+$ and $\delta_-$, respectively.

The integers $\delta_\pm$, which satisfy $0 \leq \delta_\pm \leq m$, characterize the degree of indeterminacy of the solution set of the matrix Hamburger moment problem.  
It was shown in~\cite{Kog} that if either of the numbers $\delta_\pm$ equals $m$, then both equal $m$.  
In this case, the matrix Hamburger moment problem is said to be \emph{completely indeterminate},  
and the corresponding limit matrix Weyl balls are said to be \emph{nondegenerate}.  
Otherwise, the matrix Hamburger moment problem and the associated limit matrix Weyl balls are said to be \emph{degenerate}.

Assume now that the integers $\delta_+$ and $\delta_-$ satisfy 
$0 \leq \delta_+ \leq m - 1$ and $0 \leq \delta_- \leq m - 1$, 
or $\delta_+ = \delta_- = m$.  
It follows from the results in~\cite{D1}, \cite{D2} that there exists a matrix Hamburger moment problem of order $m$ 
such that the left semi-radius of the corresponding limit matrix Weyl ball has rank $\delta_+$ 
and the right semi-radius has rank $\delta_-$.  

Given the sequence \eqref{s_j}, define the following block matrices:
\begin{equation*}
y_{j,k} = \begin{pmatrix}
s_j \\
\vdots \\
s_k
\end{pmatrix}, \quad 
z_{j,k} = (s_j \ \ldots \ s_k)
\end{equation*}
for all integers \(j\) and \(k\) such that \(0 \le j \le k\). Furthermore, for each
integer \(j>0\), consider the block matrices
\begin{equation*}
u_0 = O_{m \times m}, \quad 
u_j = \begin{pmatrix}
O_{m \times m} \\
- y_{0,j-1}
\end{pmatrix}, \quad 
v_0 = I_m,\
v_j = \begin{pmatrix}
I_m \\
O_{mj \times m}
\end{pmatrix},
\end{equation*}
\begin{equation*}
\widetilde{v}_0 = I_m,\ \
\widetilde{v}_j = \begin{pmatrix}
O_{m\times mj} & I_m 
\end{pmatrix}, \ \
T_0 = O_{m \times m}, \ \
T_j = \begin{pmatrix}
O_{m \times mj} & O_{m \times m} \\
I_{mj} & O_{mj \times m}
\end{pmatrix},
\end{equation*}
\begin{equation*}
R_{T_j}(z) = \left( I_{(j+1)m} - z T_j \right)^{-1} =
\begin{pmatrix}
I_m & O & \cdots & O \\
z I_m & I_m & \cdots & O \\
\vdots & \ddots & \ddots & \vdots \\
z^j I_m & \cdots & z I_m & I_m
\end{pmatrix}.
\end{equation*}

Define   
\[
\widehat{H}_0 = \mathbf{H}_0, \quad 
\widehat{H}_j = s_{2j} - z_{j,2j-1} \mathbf{H}_{j-1}^{-1} y_{j,2j-1}
\qquad \text{for all } j > 0.
\] 
It is known (see, for example, \cite{D0}) that  
\[
\widehat{H}_j > O_{m \times m}
 \qquad \text{for all } j \geq 0.
\]  

For all \( j \geq 0 \), we define the matrix polynomials of the first and second kind by  
\begin{align}
P_j(z) &= \widehat{H}_j^{1/2} \widetilde{v}_j \mathbf{H}_j^{-1} R_{T_j}(z) v_j, \label{P12kinB}\\
Q_j(z) &= -\widehat{H}_j^{1/2} \widetilde{v}_j \mathbf{H}_j^{-1} R_{T_j}(z) u_j. \label{Q12kinB}
\end{align}

Consider the Hermitian matrix
\begin{equation}\label{J}
\mathcal{J}= 
\begin{pmatrix}
O & -iI \\
iI & O
\end{pmatrix}
\in {\mathbb C}^{2m\times 2m}.
\end{equation}
Clearly,
\(
\mathcal{J}^2=I, \ \mathcal{J}^*=\mathcal{J}.
\)

\begin{defn}
The linear matrix polynomial
\[
b_j : \mathbb{C} \to \mathbb{C}^{2m \times 2m},
\]
defined by
\begin{equation}\label{bpfl}
b_j(z) = I_{2m} - izE_j,\qquad j \ge 0,
\end{equation}
is called a \emph{Blaschke--Potapov factor}.
Here
\begin{equation}\label{nilp}
E_j =
\begin{pmatrix}
P_j^*(0)P_j(0)& -P_j^*(0)Q_j(0)\\
- Q_j^*(0)P_j(0)&  Q_j^*(0)Q_j(0)
\end{pmatrix}
\mathcal{J}=
\begin{pmatrix}
P_j^*(0) \\
- Q_j^*(0)
\end{pmatrix}
\bigl(
P_j(0)-Q_j(0)
\bigr)
\mathcal{J}.
\end{equation}
\end{defn}
From the matrix version of the Ostrogradsky--Liouville formulas (see, e.g., \cite{Berg}), it follows that
\[
P_j(0) Q_j^*(0) - Q_j(0) P_j^*(0) = O\qquad \text{for each\ }  j \ge 0.
\]
Hence, \(E_j^2 = O\) and the representation \eqref{bpfl} can be rewritten as
\begin{equation*}
b_j(z) = \exp\bigl(- i z E_j\bigr)\qquad \text{for each\ } j \ge 0.
\end{equation*}

\begin{thm}[Hamburger Criterion]\label{BPP}
Let the matrix Hamburger moment problem \eqref{HM} be given, and suppose that
the conditions \eqref{PMH} are satisfied. Let the matrix polynomials of the
first and second kind be defined by \eqref{P12kinB} and \eqref{Q12kinB},
respectively, and let the matrices \(E_j\) be defined by \eqref{nilp}.
Then the matrix Hamburger moment problem \eqref{HM} is completely
indeterminate if and only if the matrix series
\begin{equation}\label{SumE}
\sum_{j=0}^{\infty} 
\begin{pmatrix}
P_j^*(0)P_j(0)& -P_j^*(0)Q_j(0)\\
- Q_j^*(0)P_j(0)&  Q_j^*(0)Q_j(0)
\end{pmatrix}
\mathcal{J}
\end{equation}
converges.
\end{thm}

\begin{cor}\label{R26}
Under the assumptions of Theorem~\ref{BPP}, the matrix Hamburger moment
problem \eqref{HM} is \emph{completely indeterminate} if and only if both
matrix series
\begin{equation}\label{MHC}
\sum_{j=0}^{\infty}P_j^*(0)P_j(0),
\qquad
\sum_{j=0}^{\infty}Q_j^*(0)Q_j(0)
\end{equation}
converge.
\end{cor}
In the scalar case, criterion \eqref{MHC} reduces to Hamburger's classical criterion for indeterminacy of the moment problem.
\begin{thm}\label{BPRM}
Assume that the hypotheses of Theorem~\ref{BPP} hold and that the matrix Hamburger moment problem \eqref{HM} is \emph{completely indeterminate}.
Then, for every \(z\in\mathbb{C}\), the infinite Blaschke--Potapov product
\begin{equation}\label{RMbpf}
U(z)=\mathop{\overrightarrow{\prod}}_{j=0}^{\infty}\exp\bigl(-izE_j\bigr)
\end{equation}
converges and defines an entire matrix function \(U\), where  the arrow
indicates that the factors are multiplied in increasing order of the index
\(j\), from left to right.
\end{thm}

\begin{defn}
The entire matrix function
\(
U:\mathbb{C}\to\mathbb{C}^{2m\times 2m}
\)
defined by \eqref{RMbpf} is called the \emph{Nevanlinna matrix} of the completely indeterminate moment problem \eqref{HM}.
\end{defn}

The representation \eqref{RMbpf} of the Nevanlinna matrix as a product of elementary Blaschke–Potapov matrix factors is a particular case of the general theory developed by V. P. Potapov \cite{P1}. Representations of this kind arising in interpolation problems have been studied in \cite{Ko}, \cite{KPW}.

\begin{defn}
Let $G$ be a nonempty open subset of the complex plane.
A subset $\mathcal{D} \subset G$ is called a \emph{discrete subset} of $G$ if it has no accumulation points in $G$; equivalently, every point of $\mathcal{D}$ is isolated.
\end{defn}

\begin{defn}
Let $\phi$ and $\psi$ be $m \times m$ matrix functions meromorphic in $\mathbb{C}_+$.  
The pair $\begin{pmatrix}\phi \\ \psi\end{pmatrix}$, together with a discrete subset $\mathcal{D}_{\phi\psi} \subset \mathbb{C}_+$, is called a \emph{Nevanlinna pair} if the following conditions hold: 
\begin{enumerate}
  \item $\phi$ and $\psi$ are holomorphic in $\mathbb{C}_+ \setminus \mathcal{D}_{\phi\psi}$.
  \item $\operatorname{rank} \begin{pmatrix}\phi(z) \\ \psi(z)\end{pmatrix} = m$ 
  \qquad \text{for all } $z \in \mathbb{C}_+ \setminus \mathcal{D}_{\phi\psi}$.
  \item $\begin{pmatrix}\phi(z) \\ \psi(z)\end{pmatrix}^* \mathcal{J} \begin{pmatrix}\phi(z) \\ \psi(z)\end{pmatrix} \geq O$
       \qquad \text{for all } $z \in \mathbb{C}_+ \setminus \mathcal{D}_{\phi\psi}$.
\end{enumerate}
\end{defn}
\noindent The set \( \mathcal{D}_{\phi\psi} \) is called the \emph{exceptional set} of the Nevanlinna pair
\(
\begin{pmatrix}
\phi \\
\psi
\end{pmatrix}.
\)

\begin{defn}
Two Nevanlinna pairs
\(
\begin{pmatrix}\phi_1\\[2pt]\psi_1\end{pmatrix}
\)
and
\(
\begin{pmatrix}\phi_2\\[2pt]\psi_2\end{pmatrix}
\),
with exceptional sets
$\mathcal{D}_{\phi_1\psi_1}$ and $\mathcal{D}_{\phi_2\psi_2}$, respectively,
are said to be \emph{equal} if their exceptional sets coincide,
$\mathcal{D}_{\phi_1\psi_1} = \mathcal{D}_{\phi_2\psi_2}$,
and for all $z \in \mathbb{C}_+ \setminus \mathcal{D}_{\phi_1\psi_1}$ the following equalities hold:
\[
\phi_1(z) = \phi_2(z), \qquad \psi_1(z) = \psi_2(z).
\]
\end{defn}

We denote by $\overline{\mathcal{R}}_m$ the set of all Nevanlinna pairs of $m \times m$ matrix functions.

\begin{defn}
Let $\mathcal{Q}_m$ denote the set of all meromorphic $m \times m$ matrix functions $Q(z)$ in $\mathbb{C}_+$ such that $\det Q(z) \not\equiv 0$.  
For each such function, let $\mathcal{D}_Q$ denote the discrete subset of $\mathbb{C}_+$ consisting of the isolated singularities of $Q(z)$ together with the zeros of $\det Q(z)$.  
The set $\mathcal{D}_Q$ is called the \emph{exceptional set} of the matrix function $Q \in \mathcal{Q}_m$.
\end{defn}

\begin{defn}
Two Nevanlinna pairs
\(
\begin{pmatrix}\phi_1\\[2pt]\psi_1\end{pmatrix}\)
and
\(\begin{pmatrix}\phi_2\\[2pt]\psi_2\end{pmatrix}
\)
with exceptional sets
$\mathcal{D}_{\phi_1\psi_1}$ and $\mathcal{D}_{\phi_2\psi_2}$, respectively,
are said to be \emph{equivalent} if there exists a meromorphic matrix function
\( Q \in \mathcal{Q}_m \) with the exceptional set
$\mathcal{D}_Q$
such that
\[
\begin{pmatrix}\phi_1(z)\\[2pt]\psi_1(z)\end{pmatrix}
=
\begin{pmatrix}\phi_2(z)\\[2pt]\psi_2(z)\end{pmatrix} Q(z)
\qquad
\text{for all } z \in \mathbb{C}_+ \setminus
\bigl\{\mathcal{D}_{\phi_1\psi_1} \cup \mathcal{D}_{\phi_2\psi_2} \cup \mathcal{D}_Q\bigr\}.
\]
\end{defn}

It is easy to verify that this relation is an \emph{equivalence relation}; that is, it is reflexive, symmetric, and transitive. Consequently, the set of Nevanlinna pairs is partitioned into disjoint equivalence classes. The set of all such equivalence classes is denoted by 
\(\langle\overline{\mathcal{R}}_m\rangle\).

Let $\alpha, \beta, \gamma, \delta, \phi, \psi \in \mathbb{C}^{m \times m}$ be matrices such that $\alpha \phi + \beta \psi$ is invertible.  
Then, by definition, the quotient
\[
\frac{\gamma \phi + \delta \psi}{\alpha \phi + \beta \psi}
\]
is understood as
\[
\bigl( \gamma \phi + \delta \psi \bigr)
\bigl( \alpha \phi + \beta \psi \bigr)^{-1}.
\]

\begin{thm}\label{tN2}
Suppose that the matrix Hamburger moment problem~\eqref{HM} is
\emph{completely indeterminate}, and let $\mathcal{F}$ denote the set of associated
functions defined by~\eqref{AssF}. Let the Nevanlinna matrix \(U\), defined
by~\eqref{RMbpf}, be partitioned into four \(m\times m\) blocks as follows:
\[
U=
\begin{pmatrix}
\alpha & \beta\\
\gamma & \delta
\end{pmatrix}.
\]

Then the following statements hold:
\begin{enumerate}
\item
Let  
\(
\begin{pmatrix}
\phi \\ \psi
\end{pmatrix}
\in \overline{\mathcal{R}}_m\)
be a {Nevanlinna pair} with exceptional set $\mathcal{D}_{\phi\psi}$. Define the matrix function \(w(z)\) by
\begin{equation} \label{npFull}
  w(z) =
  \begin{cases}
    \displaystyle\frac{\gamma(z)\phi(z) + \delta(z)\psi(z)}{\alpha(z)\phi(z) + \beta(z)\psi(z)}
    & \text{for each } z \in \mathbb{C}_+ \setminus \mathcal{D}_{\phi \psi}, \\[1.2em]
    \displaystyle\lim_{\substack{s \to z\\ s \in \mathcal{U}^*(z)}}
    \frac{\gamma(s)\phi(s) + \delta(s)\psi(s)}{\alpha(s)\phi(s) + \beta(s)\psi(s)}
    & \text{for each } z \in \mathcal{D}_{\phi \psi},
  \end{cases}
\end{equation}
where \(\mathcal{U}^*(z)\) denotes a punctured neighborhood of \(z \in \mathcal{D}_{\phi \psi}\) contained in \(\mathbb{C}_+\) and disjoint from \(\mathcal{D}_{\phi \psi}\).
Then the function \(w\) is well-defined and holomorphic in \(\mathbb{C}_+\); moreover, \(w \in \mathcal{F}\).

\item
Conversely, every \(w \in \mathcal{F}\) admits a representation of the form~\eqref{npFull} for some Nevanlinna pair 
\(
\begin{pmatrix}
\phi\\
\psi
\end{pmatrix}\in  \overline{\mathcal{R}}_m.
\)

\item
Two Nevanlinna pairs generate the same matrix function \(w\) via~\eqref{npFull} if and only if they are equivalent. Hence,~\eqref{npFull} establishes a one-to-one correspondence between the set \(\mathcal{F}\) and the set of equivalence classes of Nevanlinna pairs \(\langle \overline{\mathcal{R}}_m \rangle\).
\end{enumerate}
\end{thm}

Theorems~\ref{BPP}, \ref{BPRM}, and~\ref{tN2}, stated above, are fundamental results in the theory of the matrix Hamburger moment problem. The matrix Hamburger moment problem is equivalent to a certain interpolation problem for matrix Nevanlinna functions. Another example of such an interpolation problem is the matrix Nevanlinna--Pick interpolation problem.

In the remainder of this paper, we establish analogues of Theorems~\ref{BPP}, \ref{BPRM}, and~\ref{tN2} for the matrix Nevanlinna--Pick interpolation problem.

\section{The Matrix Nevanlinna–Pick Interpolation Problem}

In the \emph{matrix Nevanlinna--Pick interpolation problem}, one seeks to describe all Nevanlinna matrix functions \(w \in \mathcal{R}_m\) that satisfy the interpolation conditions
\begin{equation}
w(z_j) = w_j \qquad \text{for all } j \in \mathbb{N}.
\label{NP}
\end{equation}
Here \(\{z_j\}_{j=1}^\infty \subset \mathbb{C}_+\) is a sequence of distinct interpolation nodes, and
\(
{w_j}_{j=1}^\infty \subset \mathbb{C}^{m \times m}
\)
is the corresponding sequence of interpolation values.

We denote the sequences of interpolation nodes appearing in~\eqref{NP} and their complex conjugates by
\begin{equation*}
\mathcal{Z} = \{z_j\}_{j=1}^\infty \subset \mathbb{C}_+,
\qquad
\overline{\mathcal{Z}} = \{\overline{z}_j\}_{j=1}^\infty \subset \mathbb{C}_- .
\end{equation*}
Let \( \mathcal{F} \) denote the class of all Nevanlinna matrix functions
\( w \in \mathcal{R}_m \) satisfying the interpolation conditions~\eqref{NP}.

The matrix Nevanlinna--Pick interpolation problem and its various extensions have been extensively studied
(see, for instance, \cite{Akh}, \cite{D00}, \cite{D1}, \cite{D5}, \cite{D2}, \cite{D4}, \cite{DR}, \cite{PE}, \cite{KKY}, \cite{Ko}, \cite{KoPo}).
A comprehensive overview of the current state of the theory of interpolation problems for Nevanlinna functions
and their generalizations is presented in the monograph~\cite{Ar1}.

Alongside the full interpolation problem~\eqref{NP}, we also consider a sequence of truncated problems. 
For each natural number \(n\), the \emph{truncated matrix Nevanlinna--Pick interpolation problem} seeks to describe 
all Nevanlinna matrix functions \( w \in \mathcal{R}_m \) satisfying the first \(n\) interpolation conditions:
\begin{equation}
    w(z_j) = w_j, \qquad j = 1,\dots, n.
    \label{NPT}
\end{equation}
We denote the finite sequences of interpolation nodes appearing in~\eqref{NPT} and their complex conjugates by
\begin{equation*}
    \mathcal{Z}_n = \{z_j\}_{j=1}^n \subset \mathbb{C}_+, 
    \qquad 
    \overline{\mathcal{Z}}_n = \{\overline{z}_j\}_{j=1}^n \subset \mathbb{C}_- .
\end{equation*}
Let \( \mathcal{F}_n \) denote the class of all Nevanlinna matrix functions \( w \in \mathcal{R}_m \) 
that satisfy the interpolation conditions~\eqref{NPT}.

For any natural number \( n \), we introduce the following block matrices:
\begin{equation}
    T_n = 
    \begin{pmatrix}
        z_1^{-1}I_m &        &        \\
                    & \ddots &        \\
                    &        & z_n^{-1}I_m
    \end{pmatrix}, \
    v_n =
    \begin{pmatrix}
        I_m \\ \vdots \\ I_m
    \end{pmatrix} \in \mathbb{C}^{nm \times m},\
    u_n =
    \begin{pmatrix}
        w_1 \\ \vdots \\ w_n
    \end{pmatrix}, \label{BMtrun1}
\end{equation}
\begin{equation}    
    \mathbf{P}_n =
    T_n^{-1}
    \begin{pmatrix}
        \dfrac{w_1 - w_1^*}{z_1 - \bar{z}_1} & \cdots & \dfrac{w_1 - w_n^*}{z_1 - \bar{z}_n} \\
        \vdots & \ddots & \vdots \\
        \dfrac{w_n - w_1^*}{z_n - \bar{z}_1} & \cdots & \dfrac{w_n - w_n^*}{z_n - \bar{z}_n}
    \end{pmatrix}
    T_n^{-*}, \
    R_{T_n}(z) = (I - z T_n)^{-1}. \label{BMtrun2}
\end{equation}
Finally, we define
\begin{equation*}
\widetilde{v}_n = 
\begin{cases}
I_m, & \text{if } n = 1, \\[1mm]
\bigl( O_{m}\ \ldots\ O_{m}\ I_m \bigr) \in \mathbb{C}^{m \times nm}, & \text{if } n > 1.
\end{cases}
\end{equation*}

It is straightforward to verify the following \emph{fundamental identity}:
\begin{equation*}
    T_n \mathbf{P}_n - \mathbf{P}_n T_n^{*} = v_n u_n^{*} - u_n v_n^{*}.
\end{equation*}
This identity plays a central role in the theory of interpolation problems
(see, for example, \cite{Ko}, \cite{KoPo}, \cite{Sah}).

The block matrices \( \mathbf{P}_n \) are called the \emph{Pick matrices}.
The non-negativity condition \( \mathbf{P}_n \ge O_{nm \times nm} \) is necessary and sufficient for the existence of at least one solution to problem~\eqref{NPT}.
Throughout this paper, we restrict our attention to the \emph{completely indeterminate} case, characterized by
\begin{equation}\label{TNPci}
    \mathbf{P}_n > O_{nm \times nm}.
\end{equation}
This assumption excludes degenerate cases of problem~\eqref{NPT} (see, for instance, \cite{Ko,KoPo,Sah}).

\begin{defn}
Let
\[
U_n : \mathbb{C} \setminus \overline{\mathcal{Z}}_n \to \mathbb{C}^{2m \times 2m}
\]
be the matrix function defined by
\begin{equation}
U_n(z)
=
I
-
iz\,(v_n\ u_n)^*\,R_{T_n}^*(\overline{z})\,
\mathbf{P}_n^{-1}\,
(v_n\ u_n)\,\mathcal{J},
\label{RezTrun}
\end{equation}
where the block matrices appearing on the right-hand side of~\eqref{RezTrun} are defined in~\eqref{J}, \eqref{BMtrun1}, and~\eqref{BMtrun2}.

The matrix function \( U_n \) is called the \emph{resolvent matrix} of the completely indeterminate truncated Nevanlinna--Pick problem~\eqref{NPT}.
\end{defn}

As shown in~\cite{D22}, the matrix \( U_n(z) \) is invertible for all
\( z \in \mathbb{C} \setminus (\mathcal{Z}_n \cup \overline{\mathcal{Z}}_n) \) and satisfies the identity
\[
U_n^{-1}(z) = \mathcal{J}\, U_n^*(\overline{z})\, \mathcal{J}.
\]

\begin{lem}\label{FFMIt}
Assume that the truncated problem~\eqref{NPT} is completely indeterminate, and let the resolvent matrix \( U_n \) be defined by~\eqref{RezTrun}.
Then a Nevanlinna matrix function \( w \in \mathcal{R}_m \) solves~\eqref{NPT} if and only if it satisfies the \emph{Factorized Fundamental Matrix Inequality} (FFMI) of V.\,P.~Potapov:
\begin{equation}\label{FMI(f)}
  \binom{I}{w(z)}^*\ 
  \frac{U_n^{-*}(z)\,\mathcal{J}\,U_n^{-1}(z)}{i(\overline{z} - z)}\
 \binom{I}{w(z)}
  \geq O
  \qquad \text{for all } z \in \mathbb{C}_+ \setminus \mathcal{Z}_n.
\end{equation}
\end{lem}

\begin{proof}
See, for example,~\cite{D22}.
\end{proof}

\begin{thm}
Suppose that the truncated problem~\eqref{NPT} is completely indeterminate, and let the resolvent matrix \( U_n \) be given by~\eqref{RezTrun}.
Partition \( U_n \) into \( m \times m \) blocks as follows:
\[
U_n =
\begin{pmatrix}
\alpha_n & \beta_n \\
\gamma_n & \delta_n
\end{pmatrix}.
\]

Then:
\begin{enumerate}
\item Let 
\(
\begin{pmatrix}
\phi \\ \psi
\end{pmatrix}
\in \overline{\mathcal{R}}_m
\)
be a Nevanlinna pair with exceptional set
\( \mathcal{D}_{\phi\psi} \). Define the discrete set
\(
\mathcal{D} = \mathcal{D}_{\phi\psi} \cup \mathcal{Z}_n.
\)
Then the matrix function
\begin{equation}\label{npFall}
w(z) =
\begin{cases}
\dfrac{\gamma_n(z)\,\phi(z) + \delta_n(z)\,\psi(z)}
      {\alpha_n(z)\,\phi(z) + \beta_n(z)\,\psi(z)}
   &\quad \text{for each } z \in \mathbb{C}_+ \setminus \mathcal{D}, \\[2ex]
\displaystyle
\lim_{\substack{s \to z \\ s \in \mathcal{U}^*(z)}}
\dfrac{\gamma_n(s)\,\phi(s) + \delta_n(s)\,\psi(s)}
      {\alpha_n(s)\,\phi(s) + \beta_n(s)\,\psi(s)}
   &\quad \text{for each } z \in \mathcal{D}_{\phi\psi} \setminus \mathcal{Z}_n, \\[2ex]
\qquad w_j
   & \quad \text{for each } z = z_j \in \mathcal{Z}_n,
\end{cases}
\end{equation}
is holomorphic on \( \mathbb{C}_+ \) and satisfies \( w \in \mathcal{F}_n \).
Here \( \mathcal{U}^*(z) \) denotes a punctured neighborhood of \( z \in \mathcal{D}_{\phi\psi} \setminus \mathcal{Z}_n \) such that \( \mathcal{U}^*(z)\subset\mathbb{C}_+ \) and \( \mathcal{U}^*(z)\cap\mathcal{D}=\emptyset \).

\item Conversely, every function \( w \in \mathcal{F}_n \) admits a representation of the form \eqref{npFall}
for some Nevanlinna pair \( \begin{pmatrix}\phi \\ \psi\end{pmatrix}\in \overline{\mathcal{R}}_m \).

\item Two Nevanlinna pairs yield the same function \( w \) in~\eqref{npFall} if and only if they are equivalent.
Thus, formula~\eqref{npFall} establishes a bijection between the solution set \( \mathcal{F}_n \) and the set of equivalence classes of Nevanlinna pairs \( \langle\overline{\mathcal{R}}_m\rangle \).
\end{enumerate}
\end{thm}

\begin{proof}
The proof is based on the FFMI~\eqref{FMI(f)} and follows the same scheme as in~\cite{Ko}, \cite{KoPo}, and \cite{Sah}.
\end{proof}

\section{Multiplicative Structure of the Resolvent Matrix}

The main goal of this section is to represent the resolvent matrix \eqref{RezTrun} as a product of Blaschke–Potapov factors.

For each $j\ge 2$, we use the following block decompositions of the matrices \eqref{BMtrun1} and \eqref{BMtrun2}:
\begin{align*}
T_j &= \begin{pmatrix}
T_{j-1} & O_{(j-1)m \times m} \\
O_{m \times (j-1)m} & \widehat{T}_j
\end{pmatrix}, \quad
R_{T_j}(z) = \begin{pmatrix}
R_{T_{j-1}}(z) & O_{(j-1)m \times m} \\
O_{m \times (j-1)m} & R_{\widehat{T}_j}(z)
\end{pmatrix}, \nonumber \\
&\qquad\qquad\qquad
v_j = \begin{pmatrix}
v_{j-1} \\
I_m
\end{pmatrix}, \quad
u_j = \begin{pmatrix}
u_{j-1} \\
w_j
\end{pmatrix}.
\end{align*}
Let $\widehat{T}_1 = z_1^{-1} I_m$. Then, for every $j \ge 1$, we have
\begin{equation}\label{TRt1}
  \widehat{T}_j = z_j^{-1} I_m, \qquad 
  R_{\widehat{T}_j}(z)
  = (I_m - z \widehat{T}_j)^{-1}
  = (1 - z z_j^{-1})^{-1} I_m.
\end{equation}
Next, consider the following block decomposition of the Pick matrix $\mathbf{P}_j$:
\begin{equation*}
\mathbf{P}_j = 
\begin{pmatrix}
\mathbf{P}_{j-1} & B_j \\
B_j^* & C_j
\end{pmatrix}
\qquad \text{for each } j > 1,
\end{equation*}
where $C_j \in \mathbb{C}^{m \times m}$.

Define
\begin{equation}\label{Scom}
\widehat{K}_1 = \mathbf{P}_1, \qquad 
\widehat{K}_j = C_j - B_j^* \mathbf{P}_{j-1}^{-1} B_j
\qquad \text{for each } j > 1.
\end{equation}
A direct computation shows that
\begin{equation}\label{defKhat1}
\mathbf{P}_j = 
\begin{pmatrix}
I & O \\
B_j^* \mathbf{P}_{j-1}^{-1} & I
\end{pmatrix}
\begin{pmatrix}
\mathbf{P}_{j-1} & O \\
O & \widehat{K}_j
\end{pmatrix}
\begin{pmatrix}
I & \mathbf{P}_{j-1}^{-1} B_j \\
O & I
\end{pmatrix}\quad \text{for each } j > 1.
\end{equation}
Combining this identity with \eqref{TNPci} and \eqref{Scom}, we conclude that
\begin{equation}\label{KjHatPos}
\widehat{K}_j > O_{m \times m} \qquad \text{for all } j \in \mathbb{N}.
\end{equation}

\begin{defn}
The \emph{Blaschke--Potapov factors} are defined by
\begin{align*}
&b_1(z) = I - i z\,
(v_1\ u_1)^*
R_{\widehat T_1}^*(\bar z)\,
\widehat K_1^{-1}\,
(v_1\ u_1)\mathcal J, \\
&\text{and, for each}\ j > 1,\nonumber\\ 
&b_j(z) = I - i z\,
(v_j\ u_j)^*
\begin{pmatrix}
-\mathbf P_{j-1}^{-1} B_j \\
I
\end{pmatrix}
R_{\widehat T_j}^*(\bar z)\,
\widehat K_j^{-1}\,
\begin{pmatrix}
- B_j^* \mathbf P_{j-1}^{-1} & I
\end{pmatrix}
(v_j\ u_j)\mathcal J.
\end{align*}
\end{defn}
We define two infinite sequences of matrices
\(\{\widehat{v}_j\}_{j=1}^{\infty} \subset \mathbb{C}^{m\times m}\) and
\(\{\widehat{u}_j\}_{j=1}^{\infty} \subset \mathbb{C}^{m\times m}\) by
\begin{align}
&(\widehat{v}_1\  \widehat{u}_1)= ({v}_1\  {u}_1),\nonumber \\
&(\widehat{v}_j\  \widehat{u}_j)=
\begin{pmatrix}
-B_j^* \mathbf{P}_{j-1}^{-1} & I
\end{pmatrix}
(v_j\  u_j)
\qquad \text{for each } j>1.\qquad
\label{uvhat}
\end{align}
Then the Blaschke--Potapov factors can be written as
\begin{equation}\label{BPfM}
b_j(z)=I-iz\,
(\widehat{v}_j\  \widehat{u}_j)^*
R_{\widehat{T}_j}^*(\bar z)
\widehat{K}_j^{-1}
(\widehat{v}_j\  \widehat{u}_j)\mathcal{J},
\qquad
\text{for each } j\ge1.
\end{equation}
\begin{thm}\label{Factor}
Consider the Nevanlinna--Pick problem \eqref{NP}, and suppose that
\eqref{TNPci} holds for every $n\ge1$.
Let $(U_n)$ be the sequence of resolvent matrices defined by
\eqref{RezTrun}, and let $(b_j)$ be the sequence of Blaschke--Potapov
factors given by \eqref{BPfM}. Then, for every $n\ge1$ and every
\(
z\in\mathbb{C}\setminus\bigl(\mathcal{Z}_n\cup\overline{\mathcal{Z}}_n\bigr),
\)
we have
\begin{equation}\label{FAC}
U_n(z)=b_1(z)b_2(z)\cdots b_n(z).
\end{equation}
Moreover, each factor in \eqref{FAC} is invertible. Hence,
\begin{equation*}
U_n^{-1}(z)
=
b_n^{-1}(z)b_{n-1}^{-1}(z)\cdots b_1^{-1}(z).
\end{equation*}
\end{thm}

\begin{proof}
The proof is given in~\cite[Thm.~2.2]{D22}.
\end{proof}
Using the above notation, we define the sequence of matrices
\begin{equation}\label{PPj}
    \mathcal{P}_j = 
    \frac{|z_j|^2}{2\,\operatorname{Im} z_j}
    \begin{pmatrix}
        \widehat{v}_j^* \\
        \widehat{u}_j^*
    \end{pmatrix}
    \widehat{K}_j^{-1}
    (\widehat{v}_j, \widehat{u}_j)\mathcal{J}
    \qquad \text{for all } j \geq 1.
\end{equation}
In addition, we define the rational matrix functions of the first and second kind (see \cite{D22,D25}) by
\begin{eqnarray}
P_1(z) &=& \widehat{K}_1^{-1/2} R_{T_1}(z) v_1, \label{1kindP1} \\
P_j(z) &=& \widehat{K}_j^{-1/2}(-B_j^* \mathbf{P}_{j-1}^{-1}, I)\, R_{T_j}(z) v_j 
\qquad \text{for each } j > 1, \label{1kind} \\
Q_1(z) &=& -\widehat{K}_1^{-1/2} R_{T_1}(z) u_1, \label{2kindQ1} \\
Q_j(z) &=& -\widehat{K}_j^{-1/2}(-B_j^* \mathbf{P}_{j-1}^{-1}, I)\, R_{T_j}(z) u_j 
\qquad \text{for each } j > 1. \label{2kind}
\end{eqnarray}

\begin{lem}
Let the Nevanlinna--Pick problem \eqref{NP} be considered under assumptions \eqref{TNPci}, which are assumed to hold for all \(n \geq 1\).
Assume furthermore that the sequence of matrices \((\mathcal{P}_j)\) is defined by \eqref{PPj}; 
the sequence of Blaschke--Potapov factors \((b_j)\) is defined by \eqref{BPfM}; and
the sequences of rational matrix functions of the first and second kind, \((P_j)\) and \((Q_j)\), are defined by \eqref{1kindP1}, \eqref{1kind}, \eqref{2kindQ1}, and \eqref{2kind}, respectively.
Then the following statements hold:
\begin{enumerate}

\item For each \(j \geq 1\), the matrices \(\mathcal{P}_j\) satisfy
\begin{equation}\label{Pj2}
    \mathcal{P}_j^2 = -\mathcal{P}_j, \qquad 
    \mathcal{P}_j \mathcal{J} \geq O, \qquad \operatorname{tr}\mathcal{P}_j=-m,
    \qquad \operatorname{tr}\bigl (\mathcal{P}_j\mathcal{J}\bigr )>0.
\end{equation}

\item For each \(j \geq 1\), 
the matrices \(\mathcal{P}_j\) can be represented in terms of \((P_j)\) and \((Q_j)\):
\begin{equation}\label{PPjPQ}
    \mathcal{P}_j = 
\frac{|z_j|^2}{2\operatorname{Im} z_j}
\begin{pmatrix}
        P_j^*(0)P_j(0) &  -P_j^*(0)Q_j(0)       \\
       -Q_j^*(0)P_j(0) & Q_j^*(0)Q_j(0)
    \end{pmatrix}\mathcal{J}.
\end{equation}

\item For each \(j \geq 1\), the  Blaschke--Potapov matrix factors \eqref{BPfM} admit the representation
\begin{equation}\label{BP2}
    b_j(z)=I + \mathcal{P}_j - \zeta_j(z)\mathcal{P}_j,
\end{equation}
where \(\zeta_j(z)\) is the scalar Blaschke factor
\begin{equation}\label{bf}
    \zeta_j(z) = \frac{\overline{z}_j}{z_j}  \frac{z - z_j}{z - \overline{z}_j}.
\end{equation}
\end{enumerate}
\end{lem}

\begin{proof}
1. We compute
\begin{align*}
  \mathcal{P}_j^2 & =  \left (  \frac{|z_j|^2 }{ 2\operatorname{Im }z_j}\right )^2
   \begin{pmatrix}
        \widehat{v}_j^* \\
        \widehat{u}_j^*
    \end{pmatrix}
   \widehat{K}_j^{-1}
    (\widehat{v}_j, \widehat{u}_j) \mathcal{J}
    \begin{pmatrix}
        \widehat{v}_j^* \\
        \widehat{u}_j^*
    \end{pmatrix}
    \widehat{K}_j^{-1}
    (\widehat{v}_j, \widehat{u}_j) \mathcal{J} \\   
& = -i \left (  \frac{|z_j|^2 }{ 2\operatorname{Im }z_j}\right )^2
 \begin{pmatrix}
        \widehat{v}_j^* \\
        \widehat{u}_j^*
    \end{pmatrix}
   \widehat{K}_j^{-1}
    (\widehat{v}_j \widehat{u}_j^*  - \widehat{u}_j \widehat{v}_j^*)
    \widehat{K}_j^{-1}
    (\widehat{v}_j, \widehat{u}_j) \mathcal{J} \\
& = -i \left (  \frac{|z_j|^2 }{ 2\operatorname{Im }z_j}\right )^2
 \begin{pmatrix}
        \widehat{v}_j^* \\
        \widehat{u}_j^*
    \end{pmatrix}
   \widehat{K}_j^{-1}
    (\widehat{T}_j \widehat{K}_j-\widehat{K}_j \widehat{T}_j^*)
    \widehat{K}_j^{-1}
    (\widehat{v}_j, \widehat{u}_j) \mathcal{J} \\ 
& = -i \left (  \frac{|z_j|^2 }{ 2\operatorname{Im }z_j}\right )^2
 \begin{pmatrix}
        \widehat{v}_j^* \\
        \widehat{u}_j^*
    \end{pmatrix}
   \widehat{K}_j^{-1}
    (z_j^{-1} I_m\cdot \widehat{K}_j -\widehat{K}_j\cdot \overline{z}_j^{-1} I_m)
    \widehat{K}_j^{-1}
    (\widehat{v}_j, \widehat{u}_j) \mathcal{J} \\ 
& = - \left (  \frac{|z_j|^2 }{ 2\operatorname{Im }z_j}\right )^2
 \begin{pmatrix}
        \widehat{v}_j^* \\
        \widehat{u}_j^*
    \end{pmatrix}
   \widehat{K}_j^{-1}
    \frac{i(\bar z_j-z_j)}{|z_j|^2}\widehat{K}_j
    \widehat{K}_j^{-1}
    (\widehat{v}_j, \widehat{u}_j) \mathcal{J} \\ 
& = -  \frac{|z_j|^2 }{ 2\operatorname{Im }z_j}
 \begin{pmatrix}
        \widehat{v}_j^* \\
        \widehat{u}_j^*
    \end{pmatrix}
    \widehat{K}_j^{-1}
    (\widehat{v}_j, \widehat{u}_j) \mathcal{J}
 = -\mathcal{P}_j.
\end{align*}
In the preceding computation, the third equality follows from
\[
\widehat{v}_j \widehat{u}_j^* - \widehat{u}_j \widehat{v}_j^*
= \widehat{T}_j \widehat{K}_j - \widehat{K}_j \widehat{T}_j^*,\qquad \widehat{T}_j=z_j^{-1} I_m,
\]
(see \cite[Eq.~(2.41)]{D22}), whereas the fourth equality follows from \eqref{TRt1}.
This proves the first identity in~\eqref{Pj2}.
The non-negativity of the matrix \(\mathcal{P}_j \mathcal{J}\) follows immediately from \eqref{KjHatPos} and \eqref{PPj}.
 It was shown in~\cite[Eq.~(2.17)]{DR} that
\(
\operatorname{tr}\mathcal{P}_j = -m
\).

Next, for all \(j \geq 1\), we have
\begin{align*}
   \operatorname{tr}\bigl ( \mathcal{P}_j \mathcal{J}\bigr )&= 
   \operatorname{tr}\left ( \frac{|z_j|^2}{2\operatorname{Im} z_j}
    \begin{pmatrix}
        \widehat{v}_j^* \\
        \widehat{u}_j^*
    \end{pmatrix}
    \widehat{K}_j^{-1}
    (\widehat{v}_j, \widehat{u}_j)\right )\\
    &=\frac{|z_j|^2}{2\operatorname{Im} z_j}\left ( 
   \operatorname{tr} \bigl ( \widehat{v}_j^*\widehat{K}_j^{-1}\widehat{v}_j\bigr )+
   \operatorname{tr} \bigl ( \widehat{u}_j^*\widehat{K}_j^{-1}\widehat{u}_j\bigr )
    \right )>0.
\end{align*}
The inequality follows from \eqref{KjHatPos} and from the fact that the matrices
\(\widehat{v}_j\) and \(\widehat{u}_j\) are nonsingular  (see \cite[Eq.~(2.11)]{DR}).

2.
Combining \eqref{1kindP1}, \eqref{1kind}, \eqref{2kindQ1}, and \eqref{2kind} with \eqref{uvhat}, we obtain
\begin{equation*}
 P_j(0)= \widehat{K}_j^{-1/2}\widehat{v}_j,\qquad  
 Q_j(0)= -\widehat{K}_j^{-1/2}\widehat{u}_j
 \qquad \text{for each } j\geq 1.
\end{equation*}
It follows immediately from this and \eqref{PPj} that the representation \eqref{PPjPQ} holds.

3. We have
\begin{align*}
  b_j(z) 
  &= I - i z
     \begin{pmatrix}
        \widehat{v}_j^* \\
        \widehat{u}_j^*
     \end{pmatrix}
     R_{\widehat{T}_j}^*(\bar z)\widehat{K}_j^{-1} 
     (\widehat{v}_j \ \widehat{u}_j)\mathcal{J}  \\
  &= I - \frac{i z}{1 - z \bar z_j^{-1}}
     \begin{pmatrix}
        \widehat{v}_j^* \\
        \widehat{u}_j^*
     \end{pmatrix}\widehat{K}_j^{-1} 
     (\widehat{v}_j \ \widehat{u}_j)\mathcal{J}  \\
  &= I - \frac{i z}{1 - z \bar z_j^{-1}}
     \frac{2\operatorname{Im} z_j}{|z_j|^2}
  \left (   \frac{|z_j|^2 }{2\operatorname{Im} z_j}
     \begin{pmatrix}
        \widehat{v}_j^* \\
        \widehat{u}_j^*
     \end{pmatrix}
     \widehat{K}_j^{-1}
     (\widehat{v}_j \ \widehat{u}_j)\mathcal{J} \right ) \\
  &= I + \frac{z}{1 - z \bar z_j^{-1}}
     \frac{\bar z_j - z_j}{|z_j|^2}\mathcal{P}_j  \\
 & = I + \mathcal{P}_j +
     \left(\frac{z\bar z_j - z z_j}{z_j \bar z_j - z z_j} - 1\right)\mathcal{P}_j  \\
&= I + \mathcal{P}_j +
     \frac{z\bar z_j - z z_j - z_j \bar z_j + z z_j}{z_j \bar z_j - z z_j}
     \mathcal{P}_j  \\
  &= I + \mathcal{P}_j
     - \frac{\bar z_j}{z_j}\frac{z - z_j}{z - \bar z_j}\mathcal{P}_j \\
  &= I + \mathcal{P}_j - \zeta_j(z)\mathcal{P}_j.
\end{align*}
Here, the second equality follows from \eqref{TRt1}, and the fourth from \eqref{PPj}.
This proves \eqref{BP2}.
\end{proof}

\begin{rem}
The relation \(\mathcal{P}_j^2 = -\mathcal{P}_j\) implies that the polynomial
\(\lambda^2 + \lambda\) annihilates \(\mathcal{P}_j\).
It follows that the spectrum of \(\mathcal{P}_j\) is contained in
\(\{-1,0\}\).
Together with the identity
\(
\operatorname{tr}\mathcal{P}_j = -m,
\)
this shows that \(-1\) and \(0\) are the only eigenvalues of
\(\mathcal{P}_j\), each with multiplicity \(m\).
Therefore,
\(
\operatorname{rank}\mathcal{P}_j = m.
\)
Such matrices \(\mathcal{P}_j\) are called \emph{full-rank projectors}.
They play an important role in the theory of interpolation problems
\cite{Ko}, \cite{KoPo}.
\end{rem}

\section{Infinite Sums and Products of Matrices}

In this section, we introduce the basic definitions and state fundamental results concerning the convergence of infinite matrix sums and products.

We begin with the proofs of three auxiliary results that will be used in studying the convergence of matrix series.

Let \((X_j)\) be a sequence of matrices in
\(\mathbb{C}^{\,p\times q}\).
For each \(j\), let
\[
X_j
=
\left(x_{rs}^{(j)}\right)_{\substack{1\le r\le p\\1\le s\le q}}.
\]
Then
\[
X_j^*
=
\left(\overline{x_{sr}^{(j)}}\right)_{\substack{1\le r\le q\\1\le s\le p}}
\in
\mathbb{C}^{\,q\times p}.
\]
Direct computation gives
\begin{equation*}
X_j^*X_j
=
\left(
\sum_{k=1}^{p}
\overline{x_{kr}^{(j)}}\,x_{ks}^{(j)}
\right)_{r,s=1}^{q}.
\end{equation*}
Equivalently,
\begin{equation}\label{XsXcom}
\left(X_j^*X_j\right)_{rs}
=
\sum_{k=1}^{p}
\overline{x_{kr}^{(j)}}\,x_{ks}^{(j)},
\qquad
1\le r,s\le q.
\end{equation}
In particular, for the diagonal entries of \(X_j^*X_j\), we have
\begin{equation*}
\left(X_j^*X_j\right)_{rr}
=
\sum_{k=1}^{p}
\left|x_{kr}^{(j)}\right|^2,
\qquad
1\le r\le q.
\end{equation*}

\begin{lem}\label{lemXsX}
Let \((X_j)\) be a sequence of matrices in
\(\mathbb{C}^{\,p\times q}\).
The matrix series
\begin{equation}\label{SerXsX}
\sum_{j=1}^{\infty}X_j^*X_j
\end{equation}
converges componentwise (equivalently, with respect to any matrix norm) if and only if, for every \(r=1,\ldots,q\), the scalar series of the diagonal entries of the matrices \(X_j^*X_j\),
\begin{equation}\label{DigXsX}
\sum_{j=1}^{\infty}
\left(X_j^*X_j\right)_{rr}
=
\sum_{j=1}^{\infty}
\sum_{k=1}^{p}
\left|x_{kr}^{(j)}\right|^2,
\end{equation}
converges.
\end{lem}

\begin{proof}
The necessity follows immediately from the convergence of each matrix entry.

Conversely, suppose that the scalar series
\eqref{DigXsX} converges.
Then, for every
\[
1\le r\le q,\qquad
1\le s\le p,
\]
the scalar series
\[
\sum_{j=1}^{\infty}
\left|x_{sr}^{(j)}\right|^2
\]
also converges.

Hence, for every fixed \(k,r,s\),  each of the scalar sequences
\[
\left(x_{kr}^{(j)}\right)_{j=1}^{\infty}
\quad\text{and}\quad
\left(x_{ks}^{(j)}\right)_{j=1}^{\infty}
\]
belong to \(\ell^2(\mathbb{C})\).
Therefore,
\[
\sum_{j=1}^{\infty}
\overline{x_{kr}^{(j)}}\,x_{ks}^{(j)}
\]
is the inner product of these two sequences and hence converges absolutely.

Thus, the order of summation may be interchanged:
\[
\sum_{k=1}^{p}
\sum_{j=1}^{\infty}
\overline{x_{kr}^{(j)}}\,x_{ks}^{(j)}
=
\sum_{j=1}^{\infty}
\sum_{k=1}^{p}
\overline{x_{kr}^{(j)}}\,x_{ks}^{(j)},
\qquad
1\le r,s\le q.
\]
Using \eqref{XsXcom}, we obtain
\[
\sum_{j=1}^{\infty}
\left(X_j^*X_j\right)_{rs}
=
\sum_{j=1}^{\infty}
\sum_{k=1}^{p}
\overline{x_{kr}^{(j)}}\,x_{ks}^{(j)}
=
\sum_{k=1}^{p}
\sum_{j=1}^{\infty}
\overline{x_{kr}^{(j)}}\,x_{ks}^{(j)},
\qquad
1\le r,s\le q.
\]
Hence, every entry of the matrix series \eqref{SerXsX} has a finite limit, and, consequently,  the series converges componentwise.
\end{proof}

\begin{cor}\label{tr X}
Let \((X_j)\) be a sequence of matrices in
\(\mathbb{C}^{\,p\times q}\), and let \((\alpha_j)\) be a sequence of positive numbers.
Then the matrix series
\begin{equation}\label{1234}
\sum_{j=1}^{\infty}\alpha_j X_j^*X_j
\end{equation}
converges componentwise if and only if the scalar series
\[
\sum_{j=1}^{\infty}\alpha_j\operatorname{tr}(X_j^*X_j)
\]
converges.
\end{cor}
\begin{proof}
Applying Lemma~\ref{lemXsX} to the sequence of matrices
\(\bigl(\sqrt{\alpha_j}X_j\bigr)\), we obtain that the convergence of the
scalar series formed from the diagonal entries of the matrix series
\eqref{1234} is a necessary and sufficient condition for the
componentwise convergence of the series \eqref{1234}.

It remains to use the non-negativity of the diagonal entries of the
matrices \(\alpha_j X_j^*X_j\).
\end{proof}

\begin{cor}\label{AjBi}
Let $(A_j)$ and $(B_j)$ be two sequences of matrices in
$\mathbb{C}^{m\times m}$, and set
\[
X_j=(A_j\; B_j),
\]
for all $j\geq 1$. Then the matrix series
\[
\sum_{j=1}^{\infty}X_j^*X_j
=
\sum_{j=1}^{\infty}
\begin{pmatrix}
A_j^*A_j & A_j^*B_j\\
B_j^*A_j & B_j^*B_j
\end{pmatrix}
\]
converges if and only if both matrix series
\[
\sum_{j=1}^{\infty}A_j^*A_j
\qquad\text{and}\qquad
\sum_{j=1}^{\infty}B_j^*B_j
\]
converge.
\end{cor}
\begin{proof}
The proof follows immediately from the  identity
\[
\operatorname{tr}
\begin{pmatrix}
A_j^*A_j & A_j^*B_j\\
B_j^*A_j & B_j^*B_j
\end{pmatrix}
=
\operatorname{tr}(A_j^*A_j)
+
\operatorname{tr}(B_j^*B_j)
\]
and Corollary~\ref{tr X}.
\end{proof}

We now recall the basic definitions and results concerning the convergence of infinite matrix products.

Let $\{V_j\}_{j=1}^n$ be a finite sequence of $2m \times 2m$ matrices. The \emph{right} and \emph{left} products are defined, respectively, by
\[
\overrightarrow{\prod}_{j=1}^{n} V_j
=
V_1 V_2 \cdots V_n,
\qquad
\overleftarrow{\prod}_{j=1}^{n} V_j
=
V_n \cdots V_2 V_1.
\]

Let $\{V_j\}_{j=1}^\infty$ be an infinite sequence of $2m \times 2m$ matrices, which we abbreviate as $(V_j)$. The corresponding infinite right and left products are formally written as
\[
\overrightarrow{\prod}_{j=1}^{\infty} V_j
=
V_1 V_2 \cdots V_n \cdots,
\qquad
\overleftarrow{\prod}_{j=1}^{\infty} V_j
=
\cdots V_n \cdots V_2 V_1.
\]

The right infinite product is said to \emph{converge} to $P$ if the following conditions are satisfied:
\begin{enumerate}
    \item Each $V_j$ is invertible.
    \item The sequence of partial products
    \(
    P_n = \overrightarrow{\prod}_{j=1}^{n} V_j
    \)
    converges to $P$ as $n \to \infty$.
    \item The limit $P$ is invertible.
\end{enumerate}

If any of these conditions fails to hold, the product is said to \emph{diverge}. When all three conditions are satisfied, the matrix $P$ is called the right infinite product of $(V_j)$ and is denoted by
\[
P = \overrightarrow{\prod}_{j=1}^{\infty} V_j.
\]
The convergence of the left infinite product is defined analogously.

\begin{thm}\label{t33}
Let $(V_j)$ be a sequence of $2m \times 2m$ matrices such that the series
\[
\sum_{j=1}^{\infty} V_j
\]
converges. Then the infinite products
\[
\overrightarrow{\prod}_{j=1}^{\infty} \exp(V_j),
\qquad
\overleftarrow{\prod}_{j=1}^{\infty} \exp(-V_j)
\]
converge, and their limits are mutually inverse.
\end{thm}

\begin{proof}
See, e.g., \cite[Section~4.2]{Zol}.
\end{proof}

\begin{defn}
Let $\mathcal{P} \in \mathbb{C}^{2m \times 2m}$ satisfy
\(
\mathcal{P}\mathcal{J} \ge O_{2m \times 2m},
\)
and let $\lambda > 0$. Then the matrix
\[
\mathcal{R} = \exp(-\lambda\mathcal{P})
\]
is called a $\mathcal{J}$-\emph{modulus}. Moreover, its inverse
\[
\mathcal{R}^{-1} = \exp(\lambda\mathcal{P})
\]
is a $(-\mathcal{J})$-\emph{modulus}. Such matrices will be called $(-\mathcal{J})$-\emph{moduli}.
\end{defn}

This notion was introduced by V.~P.~Potapov \cite{P1,P3}; see also \cite[Section~2.13]{Ar1} and \cite[Section~1.4]{DFK}.

The following theorem is a direct consequence of V.~P.~Potapov's fundamental theorem on the convergence of discrete products of moduli and its continuous analogue; see \cite{P2,P3,P4} and \cite[Section~4.3]{Zol}.

\begin{thm}\label{Pmod}
Let $(\mathcal{P}_j)$ be a sequence of $2m \times 2m$ matrices such that
\[
\mathcal{P}_j \mathcal{J} \ge O,
\qquad j \ge 1,
\]
and let $(\lambda_j)$ be a sequence of positive real numbers.

For each $n \ge 1$, consider the products of $(-\mathcal{J})$-\emph{moduli}
\begin{equation*}
U_n = \overrightarrow{\prod}_{j=1}^n \exp(\lambda_j \mathcal{P}_j),
\end{equation*}
and the products of $\mathcal{J}$-\emph{moduli}
\begin{equation*}
U_n^{-1} = \overleftarrow{\prod}_{j=1}^n \exp(-\lambda_j \mathcal{P}_j).
\end{equation*}
Assume that there exists a constant $C>0$ such that, for all $n \ge 1$,
\begin{equation*}
\|U_n \mathcal{J} U_n^*\| \le C,
\qquad
\|U_n^{-*} \mathcal{J} U_n^{-1}\| \le C.
\end{equation*}
Then the series
\begin{equation*}
\sum_{j=1}^{\infty} \lambda_j \mathcal{P}_j
\end{equation*}
converges. Moreover, the infinite products
\begin{equation*}
\overrightarrow{\prod}_{j=1}^{\infty} \exp(\lambda_j \mathcal{P}_j),
\qquad
\overleftarrow{\prod}_{j=1}^{\infty} \exp(-\lambda_j \mathcal{P}_j)
\end{equation*}
converge, and their limits are mutually inverse.
\end{thm}

\section{The Hamburger Criterion }

The main goal of this section is to establish a Hamburger-type criterion for the complete indeterminacy of the matrix Nevanlinna--Pick interpolation problem. Using this criterion, we derive a representation of the resolvent matrix for a completely indeterminate matrix Nevanlinna--Pick problem in the form of an infinite product of Blaschke--Potapov  factors.

Consider the matrix Nevanlinna--Pick problem with infinitely many interpolation nodes~\eqref{NP} together with the associated sequence of truncated problems~\eqref{NPT}. We assume that the conditions for complete indeterminacy of the truncated problems~\eqref{TNPci} are satisfied for all \(n \geq 1\). The nonempty sets of solutions to problems~\eqref{NP} and~\eqref{NPT} will be denoted by \(\mathcal{F}\) and \(\mathcal{F}_n\), respectively.

Let us fix a point $z_0\in \mathbb{C}_+\setminus \mathcal{Z}$ and consider the set
\begin{equation}\label{WBn}
\mathfrak{K}_n(z_0) = \{w(z_0) : w \in \mathcal{F}_n \}. 
\end{equation}
There exist matrices $r_n(z_0) \in \mathbb{C}^{m\times m}_>$, 
$\rho_n(z_0) \in \mathbb{C}^{m\times m}_>$, and 
$c_n(z_0) \in \mathbb{C}^{m\times m}$ such that \eqref{WBn} admits the representation
\begin{equation*}
\mathfrak{K}_n(z_0) = \left\{ c_n(z_0) + r_n(z_0) V \rho_n(z_0) : V^* V \leq I_m \right\}. 
\end{equation*}
The set $\mathfrak{K}_n(z_0)$ can be interpreted as a matrix ball with \emph{center} $c_n(z_0)$, 
\emph{left semi-radius} $r_n(z_0)$, and \emph{right semi-radius} $\rho_n(z_0)$.
This set is called the \emph{matrix Weyl ball} at the point $z_0$, associated with the interpolation problem~\eqref{NPT}.

The following limits exist:
\begin{equation*}
    r(z_0)=\lim_{n\to\infty} r_n(z_0),\quad
\rho(z_0)=\lim_{n\to\infty} \rho_n(z_0),\quad
   c(z_0)=\lim_{n\to\infty} c_n(z_0).
\end{equation*}

The set
\begin{equation*}
\mathfrak{K}(z_0) = \{ w(z_0) : w \in \mathcal{F} \}
\end{equation*}
can be written in the form
\begin{equation*}
\mathfrak{K}(z_0) =
\left\{
c(z_0) + r(z_0) V \rho(z_0) : V^{*}V \leq I_m
\right\}.
\end{equation*}
It is natural at this point to introduce the following terminology:
the set $\mathfrak{K}(z_0)$ is called the \emph{limit matrix Weyl ball}
corresponding to $z_0$ for the interpolation problem~\eqref{NP},
with \emph{center} $c(z_0)$, \emph{left semi-radius} $r(z_0)$, and
\emph{right semi-radius} $\rho(z_0)$.

\begin{defn}
The matrix Nevanlinna–Pick interpolation problem \eqref{NP} is said to be \emph{completely indeterminate} if
\[
 \operatorname{rank} r(z_0) = \operatorname{rank} \rho(z_0) = m.
\]
\end{defn}
It follows from S.~A.~Orlov’s theorem~\cite{Or} that, in the completely indeterminate case,
\[
 \operatorname{rank} r(z) = \operatorname{rank} \rho(z) = m
 \qquad \text{for all }  z \in \mathbb{C}_+\setminus \mathcal{Z}.
\]

\begin{defn}
A matrix \(U \in \mathbb{C}^{2m \times 2m}\) is said to be \(\mathcal{J}\)-expansive, \(\mathcal{J}\)-unitary, or \(\mathcal{J}\)-contractive if it satisfies, respectively,
\begin{equation}\label{defEUC}
U \mathcal{J} U^{*} \ge \mathcal{J}, \quad
U \mathcal{J} U^{*} = \mathcal{J}, \quad
U \mathcal{J} U^{*} \le \mathcal{J}.
\end{equation}
\end{defn}
It is well known (see, for example, \cite{Ar1}) that, in this definition,
the conditions \eqref{defEUC} can be replaced by the following equivalent ones:
\begin{equation*}
U^{*} \mathcal{J} U \ge \mathcal{J}, \quad
U^{*} \mathcal{J} U = \mathcal{J}, \quad
U^{*} \mathcal{J} U \le \mathcal{J}.
\end{equation*}
Moreover, the product of $\mathcal J$-unitary matrices is again $\mathcal J$-unitary.
Every $\mathcal{J}$-unitary matrix $U$ is invertible; moreover,
\begin{equation}\label{JUinv}
U^{-1} = \mathcal{J} U^{*} \mathcal{J},
\end{equation}
and its inverse is also $\mathcal{J}$-unitary.

In what follows, we consider only the completely indeterminate interpolation problem~\eqref{NP}.
Our immediate goal is to obtain a formula for the resolvent matrix \(U\) of the interpolation problem~\eqref{NP} as the limit of the resolvent matrices \(U_n\) corresponding to the truncated problems~\eqref{NPT}. The main difficulty is that each resolvent matrix \(U_n\) is defined only up to right multiplication by an arbitrary \(J\)-unitary matrix. Fixing this \(J\)-unitary factor is referred to as \emph{normalizing} the resolvent matrix. For the truncated problems, the normalization of the resolvent matrix may be chosen arbitrarily. However, it is a nontrivial task to choose the normalizations of the resolvent matrices \(U_n\) associated with the truncated problems~\eqref{NPT} so that they converge to the resolvent matrix \(U\) of the interpolation problem~\eqref{NP}.

We now introduce normalizations of the resolvent matrices corresponding to the truncated problems~\eqref{NPT} so that the resulting sequence converges to the resolvent matrix of the interpolation problem~\eqref{NP}.

\begin{lem}
Let $\mathcal{P}\in \mathbb{C}^{2m\times 2m}$ satisfy
\begin{equation*}
\mathcal{P}^{2}=-\mathcal{P}, 
\qquad 
\mathcal{P}\mathcal{J}\ge O_{2m\times 2m}, 
\qquad 
\operatorname{tr}\mathcal{P}=-m,
\qquad \operatorname{tr}\bigl (\mathcal{P}\mathcal{J}\bigr )>0
\end{equation*}
Then the following assertions hold.
\begin{enumerate}

\item For every $\lambda\in\mathbb{C}$,
\begin{equation}\label{Iden1}
\exp(\lambda\mathcal{P})
= I + \bigl(1-e^{-\lambda}\bigr)\mathcal{P}.
\end{equation}

\item For every real number $\varphi$, the matrix \(\exp \bigl(i\varphi \mathcal{P}\bigr)\) is \(\mathcal{J}\)-unitary:
\begin{equation}\label{Junit}
  \mathcal{J}-\exp \bigl(i\varphi \mathcal{P}\bigr)\mathcal{J} 
     \bigl(  \exp \bigl(i\varphi \mathcal{P}\bigr)\bigr)^*=O_{2m\times 2m}.
\end{equation}

\item Let $\mathcal{U}\in \mathbb{C}^{2m\times 2m}$ be a $\mathcal{J}$-unitary matrix and define
\begin{equation}\label{DefPt}
\widetilde{\mathcal{P}}
= \mathcal{U}\mathcal{P}\mathcal{U}^{-1}.
\end{equation}
Then
\begin{equation}\label{PropPt}
\widetilde{\mathcal{P}}^{2}=-\widetilde{\mathcal{P}},
\qquad
\widetilde{\mathcal{P}}\mathcal{J}\ge O_{2m\times 2m},
\qquad
\operatorname{tr}\widetilde{\mathcal{P}}=-m,
\qquad \operatorname{tr}\bigl ( \widetilde{\mathcal{P}}\mathcal{J}\bigr )>0.
\end{equation}

\item Let $\widetilde{\mathcal P}$ be defined by \eqref{DefPt}. Then, for every $\lambda\in\mathbb{C}$, the following identity holds:
\begin{equation}\label{Iden2}
\mathcal{U}^{-1}\exp\bigl(\lambda\widetilde{\mathcal{P}}\bigr)\mathcal{U}
=\exp(\lambda\mathcal{P}).
\end{equation}

\end{enumerate}
\end{lem}

\begin{proof}

1. By induction,
\[
\mathcal P^n=(-1)^{n-1}\mathcal P,
\qquad n\ge1.
\]
Therefore,
\begin{align*}
\exp(\lambda\mathcal P)
&=
I+\sum_{n=1}^{\infty}\frac{\lambda^n}{n!}\mathcal P^n  \\
&=
I+
\left(
\lambda-\frac{\lambda^2}{2!}
+\frac{\lambda^3}{3!}
-\frac{\lambda^4}{4!}
+\cdots
\right)\mathcal P   \\
&=
I+\bigl(1-e^{-\lambda}\bigr)\mathcal P,
\end{align*}
which proves \eqref{Iden1}.

2. Since $\mathcal P\mathcal J\ge O$, the matrix
$\mathcal P\mathcal J$ is Hermitian. Hence,
\[
\mathcal P\mathcal J
=
(\mathcal P\mathcal J)^*
=
\mathcal J\mathcal P^*.
\]
Combining this identity with \eqref{Iden1}, we obtain
\begin{align*}
\mathcal J
&-
\exp(i\varphi\mathcal P)\mathcal J
\bigl(\exp(i\varphi\mathcal P)\bigr)^*
\\
&=
\mathcal J
-
\exp(i\varphi\mathcal P)\mathcal J
\bigl(I+(1-e^{-i\varphi})\mathcal P\bigr)^*
\\
&=
\mathcal J
-
\exp(i\varphi\mathcal P)\mathcal J
\bigl(I+(1-e^{i\varphi})\mathcal P^*\bigr)
\\
&=
\mathcal J
-
\exp(i\varphi\mathcal P)
\bigl(I+(1-e^{i\varphi})\mathcal P\bigr)
\mathcal J
\\
&=
\mathcal J
-
\exp(i\varphi\mathcal P)
\exp(-i\varphi\mathcal P)\mathcal J
\\
&=
\mathcal J-\mathcal J
=
O.
\end{align*}
This proves \eqref{Junit}.

3. The matrix $\widetilde{\mathcal P}$ is well defined by
\eqref{DefPt}, since every $\mathcal J$-unitary matrix is invertible.
By \eqref{DefPt} and the identity
$\mathcal P^2=-\mathcal P$, we obtain
\[
\widetilde{\mathcal P}^{2}
=
\mathcal U\mathcal P^2\mathcal U^{-1}
=
-\mathcal U\mathcal P\mathcal U^{-1}
=
-\widetilde{\mathcal P}.
\]

Furthermore, by \eqref{JUinv},
\[
\widetilde{\mathcal P}\mathcal J
=
\mathcal U\mathcal P\mathcal U^{-1}\mathcal J
=
\mathcal U\mathcal P\mathcal J\mathcal U^*
\ge
O_{2m\times2m}.
\]

Hence,
\[
\operatorname{tr}\widetilde{\mathcal P}
=
\operatorname{tr}
(\mathcal U\mathcal P\mathcal U^{-1})
=
\operatorname{tr}\mathcal P
=
-m.
\]

Finally, since
$\widetilde{\mathcal P}\mathcal J\ge O$,
we have
\[
\operatorname{tr}
(\widetilde{\mathcal P}\mathcal J)
\ge0.
\]
If
\[
\operatorname{tr}
(\widetilde{\mathcal P}\mathcal J)
=0,
\]
then
\[
\widetilde{\mathcal P}\mathcal J
=
O,
\]
because a non-negative matrix with zero trace must be the
zero matrix. Since $\mathcal J$ is invertible, it follows that
$\widetilde{\mathcal P}=O$, contradicting
\[
\operatorname{tr}\widetilde{\mathcal P}
=
-m.
\]
Therefore,
\[
\operatorname{tr}
(\widetilde{\mathcal P}\mathcal J)
>0.
\]
This proves \eqref{PropPt}.

4. It follows from \eqref{Iden1} and \eqref{PropPt} that
\begin{align*}
\mathcal U
\exp(\lambda\mathcal P)
\mathcal U^{-1}
&=
\mathcal U
\bigl(
I+(1-e^{-\lambda})\mathcal P
\bigr)
\mathcal U^{-1}
\\
&=
I
+
(1-e^{-\lambda})
\mathcal U\mathcal P\mathcal U^{-1}
\\
&=
I
+
(1-e^{-\lambda})
\widetilde{\mathcal P}
\\
&=
\exp(\lambda\widetilde{\mathcal P}).
\end{align*}
Multiplying this identity by $\mathcal U^{-1}$ on the left and by
$\mathcal U$ on the right, we obtain \eqref{Iden2}.
\end{proof}

For any nonzero complex number \(z\), the principal value of the argument and the principal value of the logarithm are defined as
\begin{equation}\label{Clog}
\operatorname{Arg} z \in (-\pi,\pi], \qquad
\operatorname{Log} z = \log|z| + i\operatorname{Arg} z.
\end{equation}

\begin{rem}
We will repeatedly use the well-known fact that if square matrices \(A\) and \(B\) commute, then
\begin{equation}\label{Li}
\exp(A+B)=\exp(A)\exp(B).
\end{equation}
\end{rem}

\begin{lem}
For each \(j \ge 1\), let \(b_j\) denote the Blaschke--Potapov matrix factor  defined by~\eqref{BP2}. Then, for every \(z\in\mathbb{C}\setminus\{z_j,\overline z_j\}\),
the following statements hold:

\begin{enumerate}
\item
The factor \(b_j\) admits the exponential representation
\begin{equation}\label{BPexp}
b_j(z)
=
\exp \bigl(
-\operatorname{Log}\zeta_j(z)\mathcal{P}_j
\bigr).
\end{equation}

\item
Define
\begin{equation}\label{alfi}
\alpha_j(z)=\log |\zeta_j(z)|,
\qquad
\varphi_j(z)=\operatorname{Arg}\zeta_j(z).
\end{equation}
Then \(b_j\) admits the factorization
\begin{equation}\label{Bma}
b_j(z) 
=
\exp \bigl(
-\alpha_j(z) \mathcal{P}_j
\bigr) \cdot 
\exp \bigl(
-i\varphi_j(z) \mathcal{P}_j
\bigr).
\end{equation}
\end{enumerate}
\end{lem}

\begin{proof}

1.  
Observe that for $z \in \mathbb{C} \setminus \{ z_j, \overline{z}_j \}$, the function $\zeta_j(z)$ is well defined and nonvanishing. Now, substituting $\lambda = -\operatorname{Log} \zeta_j(z)$ into~\eqref{Iden1} yields
\[
\exp\bigl(-\operatorname{Log} \zeta_j(z) \mathcal{P}_j\bigr)
= I+\mathcal{P}_j -\exp(\operatorname{Log}\zeta_j(z))\mathcal{P}_j \]
 \[ = I+\mathcal{P}_j - \zeta_j(z)\mathcal{P}_j = b_j(z).
\]

2.  Using \eqref{BPexp}, \eqref{Clog}, and \eqref{Li}, we obtain
\begin{align*}
b_j(z)
=&\exp\bigl(-\operatorname{Log} \zeta_j(z) \mathcal{P}_j\bigr)\\
=&\exp\bigl(- \log |\zeta_j(z)| \mathcal{P}_j-i\operatorname{Arg} \zeta_j(z) \mathcal{P}_j\bigr)\\
=&
\exp \bigl(
-\alpha_j(z) \mathcal{P}_j
\bigr) \cdot 
\exp \bigl(
-i\varphi_j(z) \mathcal{P}_j
\bigr).
\end{align*}
This implies \eqref{Bma}.  \end{proof}

Throughout this chapter, we fix a point
\begin{equation}\label{z0}
z_0 \in \mathbb{C}_+ \setminus \mathcal{Z},
\end{equation}
and define a sequence of real numbers by means of the scalar Blaschke factors~\eqref{bf}:
\begin{equation}\label{FIn}
\phi_j=\operatorname{Arg}\bigl(\zeta_j(z_0)\bigr)\in(-\pi,\pi],
\qquad j\ge1.
\end{equation}

Consider the sequence of matrices \((\mathcal{P})_j\) and the sequence of real numbers \((\phi_j)\), defined by~\eqref{PPjPQ} and~\eqref{FIn}, respectively.
Using these sequences, we recursively construct a modified sequence of matrices
\((\widetilde{\mathcal{P}}_j)\) as follows:
\begin{equation}\label{Trans}
\widetilde{\mathcal{P}}_1=\mathcal{P}_1,\qquad
\widetilde{\mathcal{P}}_j
=\mathcal{U}_{j-1}\mathcal{P}_j\mathcal{U}_{j-1}^{-1},
\qquad j>1,
\end{equation}
where
\begin{equation}\label{calU}
\mathcal{U}_j=
\mathop{\overleftarrow{\prod}}\limits_{k=1}^{j}
\exp\left(-i\phi_k\widetilde{\mathcal{P}}_k\right),
\qquad j\ge1.
\end{equation}
Each factor in the product on the right-hand side is
\(\mathcal{J}\)-unitary (see~\eqref{Junit}). Therefore, each matrix
\(\mathcal{U}_j\) is also \(\mathcal{J}\)-unitary and hence invertible.

\begin{defn}
For each \(j\ge1\), the matrix function
\begin{equation}\label{BmaM}
\widetilde{b}_j(z)=
\exp\bigl(-\alpha_j(z)\widetilde{\mathcal{P}}_j\bigr) 
\exp\bigl(-i(\varphi_j(z)-\phi_j)\widetilde{\mathcal{P}}_j\bigr), \quad  z \in \mathbb{C} \setminus
                          \{z_j, \bar z_j\}
\end{equation}
is  called a \emph{modified Blaschke--Potapov factor}.
\end{defn}

\begin{lem}
Let the sequences of Blaschke--Potapov factors \((b_j)\), matrices \(({\mathcal{U}}_j)\), and modified Blaschke--Potapov factors \((\widetilde{b}_j)\) be defined by \eqref{Bma}, \eqref{calU}, and \eqref{BmaM}, respectively. Then
\begin{align}\label{BtoBP1}
b_1(z)&=\widetilde{b}_1(z){\mathcal{U}}_{1},
\qquad z \in \mathbb{C} \setminus \{z_1, \bar z_1\}, \\
\intertext{and for each \(j>1\) we have}
b_j(z)&={\mathcal{U}}_{j-1}^{-1}\widetilde{b}_j(z){\mathcal{U}}_{j},
\qquad z \in \mathbb{C} \setminus \{z_j, \bar z_j\}.
\label{BtoBPj}
\end{align}
\end{lem}

\begin{proof}
We begin by proving \eqref{BtoBP1}. By \eqref{Bma},
\begin{align*}
\widetilde{b}_1(z){\mathcal{U}}_{1}
&=
\exp\bigl(-\alpha_1(z)\widetilde{\mathcal{P}}_1\bigr) 
\exp\bigl(-i(\varphi_1(z)-\phi_1)\widetilde{\mathcal{P}}_1\bigr)
\exp\bigl(-i\phi_1\widetilde{\mathcal{P}}_1\bigr)\\
&=
\exp\bigl(-\alpha_1(z){\mathcal{P}}_1\bigr) 
\exp\bigl(-i\varphi_1(z){\mathcal{P}}_1\bigr)\\
&=
b_1(z).
\end{align*}

We next prove \eqref{BtoBPj}. Since
\(
{\mathcal U}_j=
\exp(-i\phi_j\widetilde{\mathcal P}){\mathcal U}_{j-1},
\)
we have
\begin{align*}
{\mathcal{U}}_{j-1}^{-1}&\widetilde{b}_j(z){\mathcal{U}}_{j}\\
&=
{\mathcal{U}}_{j-1}^{-1}
\exp\bigl(-\alpha_j(z)\widetilde{\mathcal{P}}_j\bigr) 
\exp\bigl(-i(\varphi_j(z)-\phi_j)\widetilde{\mathcal{P}}_j\bigr)
\exp\bigl(-i\phi_j\widetilde{\mathcal{P}}_j\bigr)
{\mathcal{U}}_{j-1}\\
&=
{\mathcal{U}}_{j-1}^{-1}
\exp\bigl(-\alpha_j(z)\widetilde{\mathcal{P}}_j\bigr) 
\exp\bigl(-i\varphi_j(z)\widetilde{\mathcal{P}}_j\bigr)
{\mathcal{U}}_{j-1}\\
&=
{\mathcal{U}}_{j-1}^{-1}
\exp\bigl(-(\alpha_j(z)+i\varphi_j(z))\widetilde{\mathcal{P}}_j\bigr)
{\mathcal{U}}_{j-1}\\
&=
\exp\bigl(-(\alpha_j(z)+i\varphi_j(z)){\mathcal{P}}_j\bigr)\\
&=
b_j(z).
\end{align*}
Here, the first, fourth, and fifth equalities follow from \eqref{calU}, \eqref{Iden2}, and \eqref{Bma}, respectively.
\end{proof}

Assume that a completely indeterminate truncated Nevanlinna--Pick problem~\eqref{NPT} is given and that its resolvent matrix \( U_n \) is defined by~\eqref{FAC}.
It follows from \eqref{Bma} that the resolvent matrix can be written in the form
\begin{equation}\label{RMBP}
U_n(z)=\overrightarrow{\prod}_{j=1}^{n}  b_j(z)=
\overrightarrow{\prod}_{j=1}^{n} \exp\bigl(-\alpha_j(z)\mathcal{P}_j\bigr)\cdot
\exp\bigl(-i\varphi_j(z)\mathcal{P}_j\bigr),\ z\in\mathbb{C}_+\setminus \mathcal{Z}_n.
\end{equation}
Along with the resolvent matrix \( U_n \), we shall also consider, for all
\( z \in \mathbb{C}_+ \setminus \mathcal{Z}_n \), the matrix function
(see also \eqref{BmaM})
\begin{equation}\label{RMBPt}
\widetilde  U_n(z)=
\overrightarrow{\prod}_{j=1}^{n} \widetilde b_j(z)=
\overrightarrow{\prod}_{j=1}^{n} \exp\bigl(-\alpha_j(z)\widetilde{\mathcal{P}}_j\bigr)\cdot
\exp\bigl(-i(\varphi_j(z)-\phi_j)\widetilde{\mathcal{P}}_j\bigr).
\end{equation}

\begin{thm}
Assume that the matrix Nevanlinna--Pick interpolation problem~\eqref{NP}
is such that all its truncated problems~\eqref{NPT} are completely
indeterminate.
Let the sequences of \( \mathcal J \)-unitary matrices
\( (\mathcal U_n) \), the resolvent matrices
\( (U_n) \), and the matrix functions
\( (\widetilde U_n) \) be defined by
\eqref{calU}, \eqref{RMBP}, and \eqref{RMBPt}, respectively.

Then, for every \( n \ge 1 \) and all
\( z \in \mathbb C \setminus
\bigl(\mathcal Z_n \cup \overline{\mathcal Z}_n\bigr) \),
the following identities hold:
\begin{align}
U_n(z) &= \widetilde U_n(z) \mathcal U_n, \label{UandtU}\\
U_n(z)\mathcal J U_n^*(z)
&= \widetilde U_n(z)\mathcal J \widetilde U_n^*(z). \label{Jform}
\end{align}
\end{thm}

\begin{proof}
We prove \eqref{UandtU} by induction on \( n \).

For \( n = 1 \), the identity follows from \eqref{BtoBP1}:
\[
\widetilde{U}_1(z)  \mathcal{U}_1
=
\widetilde{b}_1(z)\mathcal{U}_1
=
b_1(z)
=
U_1(z).
\]

Assume that \eqref{UandtU} holds for some \( n \ge 1 \). Then, by the induction hypothesis and \eqref{BtoBPj}, we obtain
\[
\widetilde{U}_{n+1}(z)\mathcal{U}_{n+1}
=
\widetilde{U}_{n}(z)\widetilde{b}_{n+1}(z) \mathcal{U}_{n+1}
=
\widetilde{U}_{n}(z)\mathcal{U}_{n}\mathcal{U}_{n}^{-1}\widetilde{b}_{n+1}(z) \mathcal{U}_{n+1}
\]
\[
=
{U}_{n}(z){b}_{n+1}(z)
=
{U}_{n+1}(z).
\]
Hence, \eqref{UandtU} holds for all \( n \).

The identity \eqref{Jform} follows immediately from \eqref{UandtU} and the \( \mathcal{J} \)-unitarity of the matrices \( \mathcal{U}_n \).
\end{proof}
It follows from \eqref{Jform} that both matrix functions \(U_n\) and \(\widetilde U_n\) are resolvent matrices for the truncated Nevanlinna--Pick problem~\eqref{NPT}. We refer to the matrix function \(\widetilde U_n\) as the \emph{modified resolvent matrix}. 

\begin{lem}\label{BPmod}
Let the modified Blaschke--Potapov factors be defined by \eqref{BmaM}. Then, for each \(j\ge1\),
\begin{equation}\label{NMBP}
\widetilde{b}_j(z_0)
=
\exp\bigl(-\alpha_j(z_0)\widetilde{\mathcal{P}}_j\bigr),
\qquad
-\alpha_j(z_0)\widetilde{\mathcal{P}}_j\mathcal{J}
\ge O_{2m\times2m}.
\end{equation}
Consequently, the matrices \(\widetilde{b}_j(z_0)\) are \((-\mathcal{J})\)-moduli.
\end{lem}

\begin{proof}
By \eqref{BmaM}, \eqref{alfi}, and \eqref{FIn},
\[
\begin{aligned}
\widetilde{b}_j(z_0)
&=
\exp\bigl(-\alpha_j(z_0)\widetilde{\mathcal{P}}_j\bigr) 
\exp\bigl(-i(\varphi_j(z_0)-\phi_j)\widetilde{\mathcal{P}}_j\bigr) \\
&=
\exp\bigl(-\alpha_j(z_0)\widetilde{\mathcal{P}}_j\bigr) 
\exp\bigl(-i\bigl(\operatorname{Arg}\zeta_j(z_0)-\operatorname{Arg}\zeta_j(z_0)\bigr)\widetilde{\mathcal{P}}_j\bigr) \\
&=
\exp\bigl(-\alpha_j(z_0)\widetilde{\mathcal{P}}_j\bigr).
\end{aligned}
\]
Thus, the first equality in \eqref{NMBP} follows.

Furthermore,
\[
-\alpha_j(z_0)
= -\log |\zeta_j(z_0)|
= -\log\left|
\frac{\overline{z}_j}{z_j} 
\frac{z_0-z_j}{z_0-\overline{z}_j}
\right|
= -\log\left|
\frac{z_0-z_j}{z_0-\overline{z}_j}
\right|
>0,
\]
where the last inequality follows from the assumptions
\(z_0,z_j\in\mathbb{C}_+\) and \(z_0\neq z_j\).
Combining this inequality with \eqref{PropPt}, we obtain the second assertion in \eqref{NMBP}.
\end{proof}

By Lemma~\ref{BPmod}, the modified resolvent matrix admits the factorization
\begin{equation}\label{RMmod}
\widetilde{U}_n(z_0)
=
\overrightarrow{\prod}_{j=1}^{n}
\exp\bigl(-\alpha_j(z_0)\widetilde{\mathcal{P}}_j\bigr),
\qquad n\ge1.
\end{equation}
All factors on the right-hand side of \eqref{RMmod} are $(-\mathcal{J})$-modules. This fact plays a key role in the proof of the following theorem.

\begin{thm}[Hamburger Criterion]\label{HC}
Suppose that all truncated problems \eqref{NPT} associated with the matrix
Nevanlinna--Pick problem \eqref{NP} are completely indeterminate.
Let \(\bigl(\widetilde{\mathcal P}_j\bigr)\) be the sequence defined by
\eqref{Trans} and \eqref{calU}.
Then the problem \eqref{NP} is completely indeterminate if and only if the
matrix series
\begin{equation}\label{HAM}
\sum_{j=1}^{\infty}
\frac{\operatorname{Im} z_j}{|z_0-z_j|^2}
\widetilde{\mathcal P}_j\mathcal J
\end{equation}
converges.
\end{thm}

\begin{proof}
Assume first that the matrix Nevanlinna--Pick problem \eqref{NP} is completely indeterminate.

By \eqref{Jform}  and \cite[Eq.~(2.33)]{D22}, for \(n\ge1\),
\begin{align}
\widetilde{U}_n^{-*}(z_0) \mathcal{J} \widetilde{U}_n^{-1}(z_0)
&=
\begin{pmatrix}
I & -c_n^*(z_0)\\
O & I
\end{pmatrix}
\begin{pmatrix}
\rho_n^{2}(z_0) & O\\
O & -r_n^{-2}(z_0)
\end{pmatrix}
\begin{pmatrix}
I & O\\
-c_n(z_0) & I
\end{pmatrix},
\label{JI0}\\[2mm]
\widetilde{U}_n(z_0) \mathcal{J} \widetilde{U}_n^*(z_0)
&=
\begin{pmatrix}
I & O\\
c_n(z_0) & I
\end{pmatrix}
\begin{pmatrix}
\rho_n^{-2}(z_0) & O\\
O & -r_n^{2}(z_0)
\end{pmatrix}
\begin{pmatrix}
I & c_n^*(z_0)\\
O & I
\end{pmatrix}.
\label{J0}
\end{align}

It is well known that if a Nevanlinna--Pick problem is completely indeterminate, then the limits
\[
\lim_{n\to\infty}c_n(z_0),\qquad
\lim_{n\to\infty}\rho_n(z_0),\qquad
\lim_{n\to\infty}r_n(z_0)
\]
exist, and that the latter two limits are positive definite. Consequently, as \(n\to\infty\), the right-hand sides of \eqref{JI0} and \eqref{J0} converge to nonsingular matrices. 

It follows that the left-hand sides of these equalities also converge to nonsingular matrices.
Thus, there exists a constant \(C>0\) such that
\begin{equation*}
\left\|
\widetilde{U}_n^{-*}(z_0) \mathcal{J} \widetilde{U}_n^{-1}(z_0)
\right\|
\leq C,
\qquad
\left\|
\widetilde{U}_n(z_0) \mathcal{J} \widetilde{U}_n^*(z_0)
\right\|
\leq C,
\quad n\geq 1.
\end{equation*}

By \eqref{RMmod}, for every \(n\geq 1\), the modified resolvent matrix
\(\widetilde U_n(z_0)\) is a product of \((-\mathcal J)\)-moduli.
Thus, all the assumptions of Theorem~\ref{Pmod} are satisfied.
 Therefore, the series
\begin{equation}\label{HAM1}
\sum_{j=1}^{\infty}
-\alpha_j(z_0)\widetilde{\mathcal P}_j
\end{equation}
converges.

Taking the trace termwise in \eqref{HAM1}, we obtain the convergence of
the scalar series
\begin{equation*}
\sum_{j=1}^{\infty}
-\alpha_j(z_0)\operatorname{tr}\widetilde{\mathcal P}_j.
\end{equation*}
It follows from \eqref{PropPt} that
\[
\operatorname{tr}\widetilde{\mathcal P}_j=-m,
\qquad j\geq 1.
\]
Therefore, the scalar series with positive terms
\begin{equation}\label{HAM11}
\sum_{j=1}^{\infty}
-\alpha_j(z_0)
\end{equation}
converges.

By the definition of \(\alpha_j(z_0)\) in \eqref{alfi} , we have
\begin{align*}
-\alpha_j(z_0)
&= -\log |\zeta_j(z_0)|
= -\frac{1}{2}\log \left (\zeta_j(z_0) \overline{ \zeta_j(z_0)}\right )\\
&= -\frac{1}{2}\log\left(
\frac{z_0-z_j}{z_0-\overline{z}_j} \frac{\overline{z_0}-\overline{z_j}}{\overline{z_0}-z_j}
\right)
= \frac12
\log\left(
1+
\frac{4\operatorname{Im}z_0\,\operatorname{Im}z_j}
{|z_0-z_j|^2}
\right).
\end{align*}
Thus
\begin{equation}\label{Help1}
-\alpha_j(z_0)
=
\frac12
\log\left(
1+
\frac{4\operatorname{Im}z_0\,\operatorname{Im}z_j}
{|z_0-z_j|^2}
\right).
\end{equation}
Since the series \eqref{HAM11} converges,
\[
-\alpha_j(z_0)\to0,
\qquad j\to\infty .
\]
Hence, \eqref{Help1} implies
\begin{equation*}
\frac{4 \operatorname{Im}z_0\, \operatorname{Im}z_j}
{|z_0-z_j|^2}
\longrightarrow0,
\qquad j\to\infty .
\end{equation*}
Using the asymptotic relation \(\log(1+t)\sim t\) as \(t\to0\), we obtain
\begin{equation}\label{eqv}
  -\alpha_j(z_0)
\sim
\frac{2 \operatorname{Im}z_0 \operatorname{Im}z_j}
{|z_0-z_j|^2},
\qquad j\to\infty .
\end{equation}
Since \(\mathcal J\) is invertible, the series \eqref{HAM1} converges if and only if the series 
\begin{equation}\label{HAM2}
\sum_{j=1}^{\infty}
-\alpha_j(z_0) \widetilde{\mathcal P}_j\mathcal{J}
\end{equation}
converges.

We have
\begin{align*}
\widetilde{\mathcal P}_j\mathcal{J}&=
\mathcal{U}_{j-1}\mathcal{P}_j\mathcal{U}_{j-1}^{-1}\mathcal{J}
=
\mathcal{U}_{j-1}\mathcal{P}_j\mathcal{J}\mathcal{U}_{j-1}^{*}\\
&=\mathcal{U}_{j-1}
\frac{|z_j|^2}{2\operatorname{Im} z_j}
\begin{pmatrix}
        P_j^*(0)P_j(0) &  -P_j^*(0)Q_j(0)       \\
       -Q_j^*(0)P_j(0) & Q_j^*(0)Q_j(0)
    \end{pmatrix}\mathcal{U}_{j-1}^{*}\\
&=\frac{|z_j|^2}{2\operatorname{Im} z_j}\mathcal{U}_{j-1}
\begin{pmatrix}
P_j^*(0) \\
- Q_j^*(0)
\end{pmatrix}
\bigl(
P_j(0), -Q_j(0)
\bigr)\mathcal{U}_{j-1}^{*}.
\end{align*}
Consequently,
\begin{equation}\label{PJXXs}
 \widetilde{\mathcal P}_j\mathcal{J}=
 \frac{|z_j|}{\sqrt{2\operatorname{Im} z_j}}\mathcal{U}_{j-1}
\begin{pmatrix}
P_j^*(0) \\
- Q_j^*(0)
\end{pmatrix}
\bigl(P_j(0), -Q_j(0)\bigr)\mathcal{U}_{j-1}^{*}\frac{|z_j|}{\sqrt{2\operatorname{Im} z_j}}.
\end{equation}
It follows from \eqref{PJXXs} that the series \eqref{HAM2}
satisfies the assumptions of Corollary~\ref{tr X}.
Hence, by Corollary~\ref{tr X}, the matrix series \eqref{HAM2}
converges if and only if the scalar series
\begin{equation}\label{HAM3}
\sum_{j=1}^{\infty}
-\alpha_j(z_0) \operatorname{tr}\left(\widetilde{\mathcal P}_j\mathcal{J}\right)
\end{equation}
converges. Thus, if the Nevanlinna--Pick problem \eqref{NP} is completely indeterminate, then the series \eqref{HAM3} converges.

Consider also the series 
\begin{equation}\label{HAM4}
\sum_{j=1}^{\infty}
\frac{\operatorname{Im}z_j}{|z_0-z_j|^2}
\operatorname{tr}\left(\widetilde{\mathcal P}_j\mathcal{J}\right).
\end{equation}
We shall show that the series \eqref{HAM3} and \eqref{HAM4} either both converge or both diverge.
Indeed, the terms of both series are positive, and by \eqref{eqv}, we have
\[
\lim_{j\to\infty}
\frac{-\alpha_j(z_0)|z_0-z_j|^2}{\operatorname{Im}z_j}
=
2 \operatorname{Im}z_0>0 .
\]
Therefore, by the classical comparison test, the series
\eqref{HAM3} and \eqref{HAM4} either both converge or both diverge.

Thus, complete indeterminacy of the problem \eqref{NP} implies  convergence of the series \eqref{HAM4}.
Applying Corollary~\ref{tr X}  again, we conclude that the series \eqref{HAM4} converges if and only if the series \eqref{HAM} converges.

This proves the necessity of the condition \eqref{HAM}.

Conversely, assume that the series \eqref{HAM} converges. Then the series
\begin{equation*}
\sum_{j=1}^{\infty}
\frac{\operatorname{Im} z_j}{|z_0-z_j|^2}
\widetilde{\mathcal P}_j
\end{equation*}
converges and, consequently, so does the series
\begin{equation}\label{HAM6}
\sum_{j=1}^{\infty}
\frac{\operatorname{Im} z_j}{|z_0-z_j|^2}
\operatorname{tr}\widetilde{\mathcal P}_j.
\end{equation}
Since
\[
\operatorname{tr}\widetilde{\mathcal P}_j=-m,
\]
it follows that the series
\begin{equation*}
\sum_{j=1}^{\infty}
\frac{\operatorname{Im} z_j}{|z_0-z_j|^2}
\end{equation*}
converges. Hence,
\begin{equation*}
\lim_{j\to\infty}
\frac{\operatorname{Im} z_j}{|z_0-z_j|^2}=0.
\end{equation*}
It follows that
\begin{equation}\label{eqv1}
-\alpha_j(z_0)=
\frac12
\log\left(
1+
\frac{4\operatorname{Im}z_0\,\operatorname{Im}z_j}
{|z_0-z_j|^2}
\right)
\sim
\frac{2\operatorname{Im}z_0\,\operatorname{Im}z_j}
{|z_0-z_j|^2},
\qquad j\to\infty.
\end{equation}

Along with the series \eqref{HAM6}, consider the series
\begin{equation}\label{HAM9}
\sum_{j=1}^{\infty}
-\alpha_j(z_0)
\operatorname{tr}\widetilde{\mathcal P}_j.
\end{equation}
Both series have non-negative terms, and it follows from \eqref{eqv1} and the classical comparison test that the series \eqref{HAM9} converges.

Applying Corollary~\ref{tr X}, we conclude that the series
\begin{equation*}
\sum_{j=1}^{\infty}
-\alpha_j(z_0)
\widetilde{\mathcal P}_j
\end{equation*}
converges. Hence, by
Theorem~\ref{t33}, the product 
\[ \widetilde{U}_n(z_0)
=
\overrightarrow{\prod}_{j=1}^{n}
\exp\bigl(-\alpha_j(z_0)\widetilde{\mathcal{P}}_j\bigr)\]
converges at \(z_0\) to a
nonsingular matrix.

It is well known that, under the assumptions \eqref{TNPci},
the limits
\[
\lim_{n\to\infty} c_n(z_0),\qquad
\lim_{n\to\infty} \rho_n(z_0),\qquad
\lim_{n\to\infty} r_n(z_0)
\]
exist.
Passing to the limit \(n\to\infty\) in \eqref{J0}, we see that the
left-hand side converges to a nonsingular matrix. Since the two
triangular factors involving \(c_n(z_0)\) converge to nonsingular
matrices, it follows that the middle factor also converges to a
nonsingular matrix. Consequently, the limits
\(\lim_{n\to\infty}\rho_n(z_0)\) and
\(\lim_{n\to\infty}r_n(z_0)\) are nonsingular matrices.

Therefore, the matrix Nevanlinna--Pick problem \eqref{NP} is completely
indeterminate.
\end{proof}

\begin{rem}
From \eqref{HAM} and the identity
\(
\operatorname{tr}\widetilde{\mathcal{P}}_j = -m,
\)
it follows that the series
\[
\sum_{j=1}^{\infty} \frac{\operatorname{Im} z_j}{|z_0 - z_j|^2}
\]
converges for the completely indeterminate Nevanlinna--Pick problem \eqref{NP}, i.e., the interpolation nodes \(\mathcal{Z}\) satisfy the Blaschke condition in the upper half-plane \(\mathbb{C}_+\).
\end{rem}

\begin{thm}\label{ReMa}
Assume that all truncated problems \eqref{NPT} associated with the matrix
Nevanlinna--Pick problem \eqref{NP} are completely indeterminate, and let
$(\widetilde{U}_n)$ be the corresponding sequence  \eqref{RMBPt} of normalized resolvent
matrices. If, in addition, the problem \eqref{NP} is completely indeterminate,
then there exist holomorphic matrix functions
\[
U:\mathbb{C}_+\setminus \mathcal{Z} \to \mathbb{C}^{2m\times 2m}, \qquad
U^{-1}:\mathbb{C}_+\setminus \mathcal{Z}  \to \mathbb{C}^{2m\times 2m},
\]
such that
\begin{equation}\label{RezolU}
U(z)=\lim_{n\to\infty} \widetilde{U}_n(z), \qquad
U^{-1}(z)=\lim_{n\to\infty} \widetilde{U}_n^{-1}(z), \qquad z \in \mathbb{C}_+\setminus \mathcal{Z},
\end{equation}
where both convergences are uniform on compact subsets of 
$\mathbb{C}_+ \setminus \mathcal{Z}$. Moreover, $U$ and $U^{-1}$ are mutually inverse:
\[
U(z) U^{-1}(z) = U^{-1}(z) U(z) = I_{2m}, \qquad z \in \mathbb{C}_+\setminus \mathcal{Z}.
\]
\end{thm}

\begin{proof}
It follows from \cite[Thm.~3.3]{D22}, together with \eqref{Jform}, that for every
$z\in \mathbb{C}_+\setminus \mathcal{Z}$ and every $n\geq1$,
\begin{equation}
        \mathcal{J} - \widetilde U_n(z)\mathcal{J} \widetilde U_n^*(z) =
         i(z-\overline{z})\sum _{j=1}^{n}
         \left(
         \begin{array}{cc}
              P_j^*(\bar z) P_j(\bar z) & -P_j^*(\bar z) Q_j(\bar z) \\
               -Q_j^*(\bar z) P_j(\bar z) & Q_j^*(\bar z) Q_j(\bar z) \\
         \end{array}
  \right)
         \label{Jfor}
    \end{equation}
and
    \begin{equation}
        \mathcal{J} - \widetilde U_n^{-*}(z)\mathcal{J}\widetilde U_n^{-1}(z)
         = i(\overline{z} - z)\sum _{j=1}^{n}\mathcal{J}
         \left(
         \begin{array}{cc}
              P_j^*(  z) P_j(z) & -P_j^*(  z) Q_j(z) \\
               -Q_j^*(  z) P_j(z) & Q_j^*(  z) Q_j(z) \\
         \end{array}
  \right)\mathcal{J}.
         \label{Jfor-}
    \end{equation}

Since \(z\in\mathbb{C}_+\), we have
\[
i(\overline z - z)=2 \operatorname{Im} z>0.
\]
Moreover, each matrix summand in the series
\eqref{Jfor} and \eqref{Jfor-} is non-negative. Consequently, for every
\(n,k\in\mathbb N\) and every
\(z\in\mathbb{C}_+\setminus\mathcal Z\),
\begin{eqnarray*}
\widetilde U_n(z)\mathcal{J} \widetilde U_n^*(z)\geq \mathcal{J},\qquad  
\widetilde U_n(z)\mathcal J \widetilde U_n^*(z)
\le
\widetilde U_{n+k}(z)\mathcal J \widetilde U_{n+k}^*(z),
\end{eqnarray*}
\begin{eqnarray}\label{Jleq1} 
\widetilde U_n^{-*}(z)\mathcal{J}\widetilde U_n^{-1}(z)\leq \mathcal{J},\qquad
\widetilde U_n^{-*}(z)\mathcal J \widetilde U_n^{-1}(z)
\ge
\widetilde U_{n+k}^{-*}(z)\mathcal J \widetilde U_{n+k}^{-1}(z).
\end{eqnarray}
As a consequence, $(\widetilde U_n)$ is a
\emph{monotonically increasing sequence of $\mathcal{J}$-expanding matrix functions},
whereas $(\widetilde U_n^{-1})$ is a
\emph{monotonically decreasing sequence of $\mathcal{J}$-contractive matrix functions}
on $\mathbb{C}_+ \setminus \mathcal{Z}$.

The proof of the sufficiency part of Theorem~\ref{HC} shows that, under the assumptions of the present theorem, the sequence
\((\widetilde{U}_n(z_0))\) converges to a nonsingular  matrix. Consequently, the sequence
\((\widetilde{U}_n^{-1}(z_0))\) also converges to a nonsingular  matrix.

By the general properties of monotone matrix sequences (see, for example,
\cite{Ar1,Or,P1}), both sequences
\((\widetilde{U}_n)\) and
\((\widetilde{U}_n^{-1})\) converge uniformly on compact subsets of
\(\mathbb{C}_+ \setminus \mathcal{Z}\). This proves all assertions of the theorem.
\end{proof}

\begin{defn}
The matrix function
\(
U:\mathbb{C}_+\setminus \mathcal{Z}\to \mathbb{C}^{2m\times 2m},
\)
defined by~\eqref{RezolU}, is called the \emph{resolvent matrix} of the completely indeterminate Nevanlinna--Pick problem~\eqref{NP}.
\end{defn}

It follows from \eqref{RMBPt} and \eqref{RezolU} that the resolvent matrix admits the representation as an infinite product of modified Blaschke--Potapov factors:
\begin{equation*}
\widetilde U(z)=
\overrightarrow{\prod}_{j=1}^{\infty}\widetilde b_j(z)=
\overrightarrow{\prod}_{j=1}^{\infty}
\exp\bigl(-\alpha_j(z)\widetilde{\mathcal{P}}_j\bigr) 
\exp\bigl(-i(\varphi_j(z)-\phi_j)\widetilde{\mathcal{P}}_j\bigr),
\end{equation*}
for all \(z\in\mathbb{C}_+\setminus\mathcal{Z}\).

\begin{lem}
Under the assumptions of Theorem~\ref{ReMa}, for every \(n\ge1\) and every
\(z\in\mathbb{C}_+\setminus\mathcal{Z}\), the following inequality holds:
\begin{equation}\label{InF}
\widetilde U_n^{-*}(z) \mathcal{J} \widetilde U_n^{-1}(z)
\ge
U^{-*}(z) \mathcal{J} U^{-1}(z).
\end{equation}
\end{lem}

\begin{proof}
Fix \(z\in\mathbb{C}_+\setminus\mathcal{Z}\).
Passing to the limit as \(k\to\infty\) in the  inequality of~\eqref{Jleq1},
and using~\eqref{RezolU}, we obtain~\eqref{InF}.
\end{proof}

\begin{lem}
Assume that all truncated interpolation problems~\eqref{NPT} and the problem~\eqref{NP} are completely indeterminate, and let the resolvent matrix \(U\) be defined by~\eqref{RezolU}. Then a Nevanlinna matrix function \(w\in\mathcal{R}_m\) is a solution to~\eqref{NP} if and only if it satisfies V.~P.~Potapov's Factorized Fundamental Matrix Inequality (FFMI):
\begin{equation}\label{FMI(i)}
\binom{I}{w(z)}^{*}
\frac{U^{-*}(z) \mathcal J U^{-1}(z)}
{i(\overline z-z)}
\binom{I}{w(z)}
\ge O,
\end{equation}
for every \(z\in\mathbb C_{+}\setminus\mathcal Z\).
\end{lem}

\begin{proof}
Assume that \(w\in\mathcal R_m\) is a solution of problem~\eqref{NP}. Then \(w\) solves each truncated problem~\eqref{NPT}. Hence, by Lemma~\ref{FFMIt} and~\eqref{Jform}, for every \(n\ge1\),
\begin{equation}\label{FMIv}
\binom{I}{w(z)}^{*}
\frac{\widetilde U_n^{-*}(z) \mathcal J \widetilde U_n^{-1}(z)}
{i(\overline z-z)}
\binom{I}{w(z)}
\ge O
\end{equation}
for every \(z\in\mathbb C_{+}\setminus\mathcal Z_n\).

Since \(\mathcal Z_n\subset\mathcal Z\), inequality~\eqref{FMIv} holds for every \(n\ge1\) and every \(z\in\mathbb C_+\setminus\mathcal Z\). Fix \(z\in\mathbb C_+\setminus\mathcal Z\) and let \(n\to\infty\) in~\eqref{FMIv}. In view of~\eqref{RezolU}, we obtain~\eqref{FMI(i)}.

Conversely, assume that \(w\in\mathcal R_m\) satisfies~\eqref{FMI(i)}. Then, by~\eqref{InF}, inequality~\eqref{FMIv} holds for every \(n\ge1\) and every \(z\in\mathbb C_+\setminus\mathcal Z\). Fix \(n\ge1\). By continuity, inequality~\eqref{FMIv} extends to the points of \(\mathcal Z\setminus\mathcal Z_n\). Consequently, \eqref{FMIv} holds for every \(z\in\mathbb C_+\setminus\mathcal Z_n\). Applying Lemma~\ref{FFMIt}, we conclude that \(w\) solves every truncated problem~\eqref{NPT}. Therefore, \(w\) is a solution of~\eqref{NP}.
\end{proof}

\begin{thm}\label{tN2np}
Suppose that the problem~\eqref{NP} is completely indeterminate, and let the resolvent matrix \(U\) be given by~\eqref{RezolU}.
Partition \(U\) into \(m\times m\) blocks as follows:
\[
U=
\begin{pmatrix}
\alpha & \beta\\
\gamma & \delta
\end{pmatrix}.
\]

Then the following statements hold:
\begin{enumerate}
\item Let
\(
\begin{pmatrix}
\phi\\
\psi
\end{pmatrix}
\in \overline{\mathcal{R}}_m
\)
be a Nevanlinna pair with exceptional set
\(\mathcal D_{\phi\psi}\).
Define the discrete set in \(\mathbb{C}_+\) by
\[
\mathcal D=\mathcal D_{\phi\psi}\cup\mathcal Z.
\]
Then the matrix function
\begin{equation}\label{npIall}
w(z)=
\begin{cases}
\dfrac{\gamma(z)\phi(z)+\delta(z)\psi(z)}
{\alpha(z)\phi(z)+\beta(z)\psi(z)}
& \quad \text{for every } z\in\mathbb C_+\setminus\mathcal D,\\[2ex]
\displaystyle
\lim_{\substack{s\to z\\ s\in\mathcal U^*(z)}}
\dfrac{\gamma(s)\phi(s)+\delta(s)\psi(s)}
{\alpha(s)\phi(s)+\beta(s)\psi(s)}
& \quad \text{for every } z\in\mathcal D_{\phi\psi}\setminus\mathcal Z,\\[2ex]
w_j
& \quad \text{for } z=z_j\in\mathcal Z,
\end{cases}
\end{equation}
is holomorphic on \(\mathbb C_+\) and satisfies \(w\in\mathcal F\).
Here, \(\mathcal U^*(z)\) denotes a punctured neighborhood of
\(z\in\mathcal D_{\phi\psi}\setminus\mathcal Z\) such that
\(\mathcal U^*(z)\subset\mathbb C_+\) and
\(\mathcal U^*(z)\cap\mathcal D=\emptyset\).

\item Conversely, every matrix function \(w\in\mathcal F\) admits a representation of the form~\eqref{npIall} for some Nevanlinna pair
\(
\begin{pmatrix}
\phi\\
\psi
\end{pmatrix}
\in \overline{\mathcal{R}}_m.
\)

\item Two Nevanlinna pairs yield the same matrix function \(w\) in~\eqref{npIall} if and only if they are equivalent.
Thus, formula~\eqref{npIall} establishes a bijection between the solution set \(\mathcal F\) and the set of equivalence classes of Nevanlinna pairs \(\langle\overline{\mathcal R}_m\rangle\).
\end{enumerate}
\end{thm}

\begin{proof}
The proof is based on the FFMI~\eqref{FMI(i)} and follows the same scheme as in
\cite{Ko,KoPo,Sah}.
\end{proof}
\begin{rem}
Theorems~\ref{HC}, \ref{ReMa}, and~\ref{tN2np}, stated above, are fundamental results in the theory of the matrix Nevanlinna--Pick interpolation problem. They are counterparts of Theorems~\ref{BPP}, \ref{BPRM}, and~\ref{tN2}, respectively, in the theory of the matrix Hamburger moment problem. It is worth noting that, in the case of the moment problem, the corresponding theorems are considerably simpler than their counterparts in the general setting of the Nevanlinna--Pick interpolation problem.
\end{rem}

\section{Matrix Nevanlinna--Pick Problem with Purely Imaginary Interpolation Nodes}
The Hamburger criterion~\eqref{HAM} for the matrix Nevanlinna--Pick problem is an analogue of the Hamburger criterion~\eqref{SumE} for the matrix Hamburger moment problem. However, the terms of the series~\eqref{HAM} have a considerably more complicated structure than the corresponding expression~\eqref{SumE}. This additional complexity stems from the fact that the interpolation nodes in the Nevanlinna--Pick problem may be located arbitrarily in the upper half-plane.

In this section, we consider the Nevanlinna--Pick  interpolation problem with purely imaginary interpolation nodes. More precisely, in the {\it special matrix Nevanlinna--Pick interpolation problem}, 
one seeks to describe all Nevanlinna matrix functions \(w \in \mathcal{R}_m\) that satisfy the interpolation conditions
\begin{equation}
w(z_j) = w_j \qquad \text{for all } j \in \mathbb{N}.
\label{NPS}
\end{equation}
Here, $(z_j) \subset \mathbb{C}_+$ is the sequence of distinct purely imaginary interpolation nodes,
\begin{equation}\label{nod+}
\mathcal{Z}=\bigl\{z_j=iy_j\bigr\}_{j=1}^{\infty},
\qquad
y_j>0,
\qquad
y_{j+1}>y_j,
\qquad
\lim_{j\to\infty}y_j=+\infty.
\end{equation}
and $(w_j)_{j=1}^\infty \subset \mathbb{C}^{m \times m}$ is the corresponding sequence of interpolation values.

The special Nevanlinna--Pick problem considered here is a particular case of the general Nevanlinna--Pick problem. Consequently, all the results established in the previous section remain valid in the present setting. On the other hand, the special problem also admits several additional results that have no counterparts in the general case. These additional results make the analogy between the moment problem and the special Nevanlinna--Pick problem substantially deeper than in the general setting.

In the previous section, we fixed an arbitrary point $z_0$ satisfying~\eqref{z0}. In the case of the special Nevanlinna--Pick problem, we now choose a positive number $y_0$ such that
\[
0<y_0<y_j,
\qquad
j\ge1,
\]
and define
\begin{equation}\label{z00}
z_0=iy_0.
\end{equation}
Throughout this section, we assume that $z_0$ is chosen according to~\eqref{z00}.

\begin{lem}\label{PtPL}
Let the special matrix Nevanlinna--Pick problem~\eqref{NPS} be such that all its truncated problems~\eqref{NPT} are completely indeterminate. Furthermore, let the sequence of matrices $(\mathcal{P}_j)$ and the sequence of modified matrices $(\widetilde{\mathcal{P}}_j)$ be defined by~\eqref{PPjPQ} and~\eqref{Trans}, respectively. Then
\begin{equation}\label{PtP}
\widetilde{\mathcal{P}}_j=\mathcal{P}_j,
\qquad
j\ge1.
\end{equation}
\end{lem}

\begin{proof}
By~\eqref{FIn}, for every $j\ge1$,
\[
\phi_j
=
\operatorname{Arg}\bigl(\zeta_j(iy_0)\bigr)
=
\operatorname{Arg}\biggl(\frac{-iy_0}{iy_0}\,\frac{iy_0-iy_j}{iy_0+iy_j}\biggr)
=
\operatorname{Arg}\biggl(\frac{y_j-y_0}{y_j+y_0}\biggr)
=
0.
\]
Here, the last equality follows from the fact that
\[
\frac{y_j-y_0}{y_j+y_0}>0,\qquad j\ge 1.
\]

Hence, by~\eqref{calU},
\[
\mathcal{U}_j=I,
\qquad
j\ge1.
\]
It then follows from~\eqref{Trans} that~\eqref{PtP} holds.
\end{proof}

The Hamburger criterion~\eqref{HC} now takes a considerably simpler form.

\begin{thm}[Hamburger criterion for the special Nevanlinna--Pick problem]\label{HCS}
Suppose that all the truncated problems~\eqref{NPT} associated with the special matrix Nevanlinna--Pick problem~\eqref{NPS} are completely indeterminate. Furthermore, let $(P_j)$ and $(Q_j)$ denote the sequences of rational matrix functions of the first and second kind, respectively, defined by~\eqref{1kindP1}, \eqref{1kind}, \eqref{2kindQ1}, and~\eqref{2kind}.

Then the problem~\eqref{NPS} is completely indeterminate if and only if the series
\begin{equation}\label{HAMS}
\sum_{j=1}^{\infty}
\begin{pmatrix}
    P_j^*(0)P_j(0) & -P_j^*(0)Q_j(0) \\
    -Q_j^*(0)P_j(0) & Q_j^*(0)Q_j(0)
\end{pmatrix}
\mathcal{J}
\end{equation}
converges.
\end{thm}

\begin{proof}
By~\eqref{nod+}, \eqref{z00}, and Lemma~\ref{PtPL}, the convergence condition in the Hamburger criterion~\eqref{HAM} is equivalent to the convergence of the series
\[
\sum_{j=1}^{\infty}
\frac{y_j}{(y_j-y_0)^2} 
\mathcal{P}_j.
\]
Using the representation~\eqref{PPjPQ} for the matrices $\mathcal{P}_j$, this series can be written in the form
\begin{equation}\label{HAMS+}
\sum_{j=1}^{\infty}
\frac{y_j^2}{2(y_j-y_0)^2}
\begin{pmatrix}
    P_j^*(0)P_j(0) & -P_j^*(0)Q_j(0) \\
    -Q_j^*(0)P_j(0) & Q_j^*(0)Q_j(0)
\end{pmatrix}
\mathcal{J}.
\end{equation}
Thus, the special Nevanlinna--Pick problem~\eqref{HCS} is indeterminate if and only if the series~\eqref{HAMS+} converges. For the proof of our theorem, it suffices to show that the series~\eqref{HAMS} and~\eqref{HAMS+} either both converge or both diverge.

Consider two numerical series with positive terms:
\begin{equation}\label{HAMS1}
\sum_{j=1}^{\infty}
\operatorname{tr}
\begin{pmatrix}
    P_j^*(0)P_j(0) & -P_j^*(0)Q_j(0) \\
    -Q_j^*(0)P_j(0) & Q_j^*(0)Q_j(0)
\end{pmatrix}
\end{equation}
and
\begin{equation}\label{HAMS+1}
\sum_{j=1}^{\infty}
\frac{y_j^2}{2(y_j-y_0)^2}
\operatorname{tr}
\begin{pmatrix}
    P_j^*(0)P_j(0) & -P_j^*(0)Q_j(0) \\
    -Q_j^*(0)P_j(0) & Q_j^*(0)Q_j(0)
\end{pmatrix}.
\end{equation}
Since
\[
\lim_{j\to\infty}
\frac{y_j^2}{2(y_j-y_0)^2}
=
\frac12
\neq0,
\]
the comparison criterion implies that the series~\eqref{HAMS1} and~\eqref{HAMS+1} either both converge or both diverge. By Corollary \eqref{tr X}, the same conclusion holds for the matrix series
\begin{equation*}
\sum_{j=1}^{\infty}
\begin{pmatrix}
    P_j^*(0)P_j(0) & -P_j^*(0)Q_j(0) \\
    -Q_j^*(0)P_j(0) & Q_j^*(0)Q_j(0)
\end{pmatrix}
\end{equation*}
and
\begin{equation*}
\sum_{j=1}^{\infty}
\frac{y_j^2}{2(y_j-y_0)^2}
\begin{pmatrix}
    P_j^*(0)P_j(0) & -P_j^*(0)Q_j(0) \\
    -Q_j^*(0)P_j(0) & Q_j^*(0)Q_j(0)
\end{pmatrix}.
\end{equation*}
It follows from this and from the nondegeneracy of the matrix \(\mathcal{J}\) that the series~\eqref{HAMS} and~\eqref{HAMS+} either both converge or both diverge.
\end{proof}

\begin{thm}
Under the assumptions of Theorem~\ref{HCS}, the special matrix Nevanlinna--Pick problem~\eqref{NPS} is completely indeterminate if and only if both series
\begin{equation}\label{MHC+}
\sum_{j=1}^{\infty}P_j^*(0)P_j(0),
\qquad
\sum_{j=1}^{\infty}Q_j^*(0)Q_j(0)
\end{equation}
converge.
\end{thm}
\begin{proof}
According to Theorem~\ref{HCS}, the special Nevanlinna--Pick problem is completely indeterminate if and only if the matrix series~\eqref{HAMS} converges.
The matrix series \eqref{HAMS} converges if and only if
the matrix series
\begin{equation}\label{BHC}
\sum_{j=1}^{\infty}
\begin{pmatrix}
    P_j^*(0)P_j(0) & -P_j^*(0)Q_j(0) \\
    -Q_j^*(0)P_j(0) & Q_j^*(0)Q_j(0)
\end{pmatrix}
\end{equation}
converges.
By Corollary~\ref{AjBi}, the convergence of the matrix series
\eqref{BHC} is equivalent to the convergence of the two matrix series
in \eqref{MHC+}.
\end{proof}
\begin{rem}
Formulas~\eqref{HAMS} and~\eqref{MHC+} are precisely the counterparts of formulas~\eqref{SumE} and~\eqref{MHC}, respectively. Indeed, they are obtained by replacing the matrix polynomials of the first and second kind with the corresponding rational matrix functions of the first and second kind. This simple correspondence holds for the problem~\eqref{NPS} but does not extend to general matrix Nevanlinna--Pick problems.
\end{rem}

\end{document}